\documentclass[10pt]{amsart}
 \usepackage[foot]{amsaddr}
\usepackage{amssymb}
\usepackage{enumerate,multicol}
\usepackage{varioref}
\usepackage{pifont}
\usepackage{dsfont}
 \usepackage{tikz}
 \usetikzlibrary{calc}
 \usetikzlibrary{arrows}
\usepackage[cmtip,all]{xy}
\usepackage{geometry}
\usepackage{colonequals}
\usepackage{multicol}

\usepackage{amsmath,amsthm}
\usepackage{amsbsy}
\usepackage{amsfonts}
\usepackage{amstext}
\usepackage{MnSymbol}
\usepackage{mathbbol}
\usepackage{wasysym}
\usepackage{eucal}
\usepackage{mathtools}
\usepackage{stackrel}

\usepackage{tikz, tikz-cd}
   \usetikzlibrary{matrix}
   \usetikzlibrary{calc}
   \usetikzlibrary{decorations}

\usepackage{float}
\usepackage{subfigure}
\usepackage{pdflscape}

\allowdisplaybreaks

\usepackage{color}
\usepackage{multirow,bigdelim}

\newtheorem{definition}[equation]{Definition}
\newtheorem{remark}[equation]{Remark}

\newtheorem{theorem}{Theorem}[section]
\newtheorem{proposition}[theorem]{Proposition}
\newtheorem{corollary}[theorem]{Corollary}
\newtheorem{lemma}[theorem]{Lemma}
\theoremstyle{definition}
\theoremstyle{remark}
\newtheorem{Notation}{Notation}[theorem]
\newtheorem{example}{Example}[theorem]

\newcommand\CC{\mathbb{C}}

\DeclareMathOperator\Hom{Hom}
\DeclareMathOperator\im{Im}
\DeclareMathOperator\HH{HH}

\newcommand{\dalia}[1]{{\ #1}}

\subjclass[2010]{Primary 16W50, 16D50; Secondary: 16E10, 16T15}

\date{}

\dedicatory{}

\title{Gerstenhaber structure on Hochschild cohomology of toupie algebras}

\author{Dalia Artenstein}
\address{}
\curraddr{Instituto de Matem\'atica y Estad\'\i stica ``Rafael Laguardia",
Facultad de Ingenier\'\i a, Universidad de la Rep\'ublica,
Julio Herrera y Reissig 565, Montevideo, Uruguay.}
\email{darten@fing.edu.uy}

\author{Marcelo Lanzilotta}
\address{}
\curraddr{Instituto de Matem\'atica y Estad\'\i stica ``Rafael Laguardia",
Facultad de Ingenier\'\i a, Universidad de la Rep\'ublica,
Julio Herrera y Reissig 565, Montevideo, Uruguay.}
\email{marclan@fing.edu.uy}

\author{Andrea Solotar}

\address{}

\curraddr{IMAS and Dto. de Matem\'{a}tica, Facultad de Ciencias Exactas y Naturales,
	Universidad de Buenos Aires, Ciudad Universitaria, Pabell\'{o}n 1, 1428, Buenos Aires,
	Argentina.}
\email{asolotar@dm.uba.ar}
\thanks{This work has been supported by the projects  UBACYT 20020130100533BA, PIP-CONICET
	112--201501--00483CO and PICT $2015_0366$ and MATHAMSUD-REPHOMOL. The third named author, is a
	research member of CONICET (Argentina).}

\thanks{}

\subjclass[2010]{Primary 16E40, 18G10; Secondary: 16D40, 16D90.}

\date{}

\begin{document}

\maketitle

\begin{abstract}
   {
    We study homological properties of a family of algebras called toupie algebras. Our main objective is to \dalia{obtain} the 
    Gerstenhaber structure of their Hochschild cohomology, with the purpose of describing the Lie algebra structure of the first 
    Hochschild cohomology space, together with the Lie module structure of the whole Hochschild cohomology.   
      }
\end{abstract}

\medskip

\textbf{Keywords:} Hochschild cohomology, Gerstenhaber algebra.


\section{Introduction}
\label{s:Intro}

 In this article we study homological properties of toupie algebras, first defined in \cite{CDHL}.
Toupie algebras combine features of canonical algebras with monomial algebras. Canonical algebras were introduced by Ringel 
in \cite{Ri}, see also \cite{BKL} for historical references about canonical algebras.

An algebra is toupie if it is a quotient of the path algebra of a finite quiver $Q$ which
has a source $0$, a sink $\omega$ and branches going from $0$ to $\omega$ by an ideal $I\subseteq Q_{\geq 2}$ generated by a set
containing two types of relations: monomial ones, which involve arrows of one branch each, and linear combinations of branches.

Canonical algebras are part of a more general class, the concealed-canonical algebras, see \cite{LP}. Since one of the properties 
distinguishing canonical algebras within the class of concealed-canonical algebras is that their quiver has only one sink and only 
one source \cite{Ri}, we
conclude that toupie algebras and concealed-canonical algebras only share the subfamily of canonical algebras.

Almost all toupie algebras are of wild representation type, see \cite{Art}. Toupie algebras are also special multiserial algebras, 
see \cite{GSc} for the definition, which are usually of wild representation type too. As a consequence, modules over toupie algebras 
are multiserial, that is non necessarily  direct finite sums of uniserial modules. The Hochschild cohomology of special multiserial 
algebras is still unknown except for some particular examples.

The Hochschild cohomology $\HH^{*}(A)$ of an algebra $A$, together with its associative algebra structure given by the cup product and 
its Gerstenhaber algebra structure,
is a derived invariant. Even having an explicit description of $\HH^{*}(A)$, the cup product and the Gerstenhaber bracket are 
not easy to compute. Bustamente proved in \cite{Bus} that the cup product of $\HH^{*}(A)$ is trivial for any triangular quadratic 
string algebra $A$, and the Gerstenhaber bracket vanishes for elements in cohomological degrees greater than $1$ when $A$ is a gentle triangular algebra. These are tame algebras with a particularly easy resolution.
Also, Redondo and Rom\'{a}n computed in \cite{RR1} the Gerstenhaber structure of the Hochschild cohomology of a triangular 
string algebra, showing that it is trivial in degrees greater than $1$. Subsequently, they  computed in \cite{RR2} the 
Gerstenhaber structure of the Hochschild cohomology of string quadratic algebras. In this case they gave conditions on the quiver 
associated to the string quadratic algebra in order to get non trivial cup product and Gerstenhaber bracket in degrees greater than $1$.
The first Hochschild cohomology space of an algebra $A$ is always a Lie algebra with the Gerstenhaber bracket.
Strametz \cite{St} described the Lie algebra structure of $\HH^{1}(A)$ for $A$ a finite dimensional monomial algebra.
Moreover, S\'{a}nchez-Flores made explicit in \cite{S} the Lie module structure of higher cohomology spaces 
$\HH^{n}(A)$ over the Lie algebra $\HH^{1}(A)$ when $A\simeq \mathds{k}Q/\langle Q_{2}\rangle$ 
--that is, radical square zero-- and $Q$ is either an oriented cycle of length $n$ or a finite quiver with no 
cycles.

In this article we describe the Lie algebra structure of $\HH^{1}(A)$, when $A$ is a toupie algebra, as well as the Lie module structure of $\HH^{n}(A)$ over $\HH^{1}(A)$.
For this, we construct a resolution of $A$ as $A$-bimodule, using technics of \cite{CS}. We shall see that the existence of non 
monomial relations will only have an effect in degrees $0, 1, 2$ of the resolution but this difference will considerably change 
the Lie structure of the first cohomology group.

Even if the dimensions of the $\mathds{k}$-vector spaces $\HH^{*}(A)$ are already known \cite{GL}, an explicit computation of 
these is needed for the description of the Gerstenhaber structure.

\medskip

The contents  of the article are as follows.
In Section \ref{Preliminaries} we fix notations and prove some preliminary results.
Section \ref{calculocohomo} is devoted to the computation of a $\mathds{k}$-basis of each Hochschild cohomology space, while 
in Section \ref{comp} we obtain the comparison morphisms between the reduced bar resolution and ours.

In Section \ref{Gerstenhaber} we prove that the Gerstenhaber bracket is zero in degrees greater that $1$, and we compute it 
when restricted to $\HH^1(A)$.

The description of $\HH^1(A)$ as a Lie algebra is given in Section \ref{Lie}, where we find necessary and sufficient conditions for 
it to be abelian and to be semisimple. We describe its centre and we prove that when $\mathds{k}=\CC$, it has a Lie subalgebra
isomorphic to $sl_a(\CC)$, where $a$ is the number of arrows from $0$ to $\omega$ in the quiver $Q$.
Finally, we prove  Theorem \ref{lie}, one of our main theorems.

In Sections \ref{H2comorep} and \ref{Hncomorep} we describe the Lie module structure of $\HH^n(A)$ for $n\ge 2$. The main results are 
Theorem \ref{descH2} and Theorem \ref{mainHn}. We end the article with an example.


\section{Preliminaries}\label{Preliminaries}

Let $\mathds{k}$ be a field of characteristic zero.

In this section we will recall some definitions, as well as some preliminary results.
\label{s:Preliminaries}
\subsection{$E$-reduced Bar resolution}
Given a finite dimensional $\mathds{k}$-algebra $A$ with radical $r$ such that $A=E\oplus r$ with $E$ a separable 
$\mathds{k}$-subalgebra $E\subset A$, Cibils proved in \cite[Lemma 2.1]{Cib} that the following complex is a projective $A$-bimodule
resolution of $A$, see also \cite{GS}
\[ \dots A\otimes_{E}\overline{A}^{\otimes_{E} 3}\otimes_{E}A  \stackrel{\overline{\delta_{2}}}{\longrightarrow} A\otimes_{E}\overline{A}^{\otimes_{E} 2}\otimes_{E} A  \stackrel{\overline{\delta_{1}}}{\longrightarrow} A\otimes_{E}\overline{A}\otimes_{E}A \stackrel{\overline{\delta_{0}}}{\longrightarrow} A\otimes_{E} A
  \stackrel{\overline{\delta_{-1}}}{\rightarrow} A  \rightarrow 0\]
where $\overline{A}=A/E$ and $\overline{\delta_{i}}$ is such that 
$\overline{\delta_{i}}(a_{0}\otimes_{E} \overline{a_{1}}\otimes_{E}\dots \otimes_{E} \overline{a_{i+1}}\otimes_{E} a_{i+2} )=
\delta_{i}(a_{0}\otimes a_{1}\otimes\dots \otimes a_{i+1}\otimes a_{i+2})$.  
From now on we will simply call it $\delta_{i}$ instead of $\overline{\delta_{i}}$.

 


We will give a contracting homotopy  $t=\{t_{i}\}_{i\geq -1}$ that we are going to use later to define the comparison maps
between this resolution and the minimal one:
\[t_{i}:A\otimes_{E}\overline{A}^{\otimes_{E}i}\otimes_{E}A\to A\otimes_{E}\overline{A}^{\otimes_{E}i+1}\otimes_{E}A\]
such that $\overline{\delta_{i+1}}\circ t_{i+1}+t_{i}\circ \overline{\delta_{i}}=id$.
We define $t_{i}(a_{0}\otimes_{E}\overline{a_{1}}\otimes_{E}\dots \otimes_{E}\overline{a_{i}}\otimes_{E}a_{i+1})=
1\otimes_{E}\overline{a_{0}}\otimes_{E}\overline{a_{1}}\otimes_{E}\dots \otimes_{E}\overline{a_{i}}\otimes_{E}a_{i+1}$.

This projective resolution will be called the $E$-reduced Bar resolution of $A$, $C_{BarE}(A)$ To simplify notation we will 
still denote by  $\delta_{n}$ the differentials of this complex.
It is possible to define the reduced cup product $\smile_{red}$ and the reduced Gerstenhaber bracket $[\quad,\quad]_{red}$
in terms of this resolution. Applying the functor  
$\Hom_{A^{e}}(-,A)$ 
to the $E$-reduced Bar resolution and considering the 
isomorphism
\[F_{X}:\Hom_{A^{e}}(A\otimes_{E} X\otimes_{E} A,A)\to \Hom_{E^{e}}(X,A)\] 
natural in $X$ given by 
\[F_{X}(\varphi)(x)=\varphi(1\otimes_{E} x\otimes_{E} 1),\] 
we obtain the following complex:
\[0\longrightarrow \Hom_{E^{e}}(E,A) \stackrel{\epsilon_{0}}{\longrightarrow} \Hom_{E^{e}}(\overline{A},A)\stackrel{\epsilon_{1}}{\longrightarrow}\Hom_{E^{e}}(\overline{A}\otimes_{E}\overline{A},A)\stackrel{\epsilon_{2}}
{\longrightarrow}\Hom_{E^{e}}(\overline{A}^{\otimes_{E}3},A) \longrightarrow \dots\]
The reduced cup product in $\bigoplus_{i=0}^{\infty}\Hom_{E^{e}}(\overline{A}^{\otimes_{E} i},A)$ is as follows: 
consider $\varphi\in \Hom_{E^{e}}(\overline{A}^{\otimes_{E} n},A)$  and $\phi\in \Hom_{E^{e}}(\overline{A}^{\otimes_{E} m},A)$, 
given $a_{1}, \dots ,a_{n+m}\in A$
\[\varphi\smile_{red}\phi(\overline{a_{1}}\otimes \dots\otimes \overline{a_{n+m}})=\varphi(\overline{a_{1}}\otimes\dots\otimes \overline{a_{n}})\phi(\overline{a_{n+1}}\otimes\dots \otimes \overline{a_{n+m}}).\] 
It is easy to prove that $\smile_{red}$ induces a product in Hochschild cohomology that coincides with the usual cup product, 
see remarks 2.3.14 and 2.3.19 of \cite{arten}.




Analogous modifications hold for the Gerstenhaber bracket, see Appendix B of \cite{San}. 

%
%
\subsection{Toupie algebras}
\begin{definition}
We recall some definitions from \cite{arten}.
A finite quiver is \emph{toupie} if it has a unique source and a unique sink and any other vertex 
is the source of exactly one arrow and the target of exactly one arrow. The source and the sink will be denoted $0$ and 
$\omega$, respectively.

Given a toupie quiver $Q$ and any admissible ideal $I\subset \mathds{k}Q$, $A= \mathds{k}Q/I$ will be called a 
\emph{toupie algebra}, the paths from $0$ to $\omega$ will be called \emph{branches}. The $j$-th branch will be denoted 
$\alpha^{(j)}$. The \emph{length} of a branch $\alpha^{(j)}$ --denoted by $|\alpha^{(j)}|$-- is the number of arrows in it.
\end{definition}

We have already observed that canonical algebras are toupie algebras. Besides canonical algebras there are many examples of toupie 
algebras, let us just mention the $n$-Kronecker algebras.

\medskip

From now on, $A= \mathds{k}Q/I$ will always denote a toupie algebra. There are four possible kinds of branches in $Q$ and we will 
consider the following order within the branches. First, let $Z=\{\alpha^{(1)},\dots ,\alpha^{(a)}\}$ be the set of arrows from 
$0$ to $\omega$. The branches $\alpha^{(a+1)}, \dots ,\alpha^{(a+l)}$ will be those of length greater than or equal to $2$ not 
involved in any relation. Next, $\alpha^{(a+l+1)},\dots ,\alpha^{(a+l+m)}$ will be the branches containing monomial relations and 
finally $\alpha^{(a+l+m+1)}, \dots ,\alpha^{(a+l+m+n)}$ will be the branches involved in non
monomial relations.

Denoting by $e_{x}$ the idempotent of $A$ corresponding to the vertex $x$, we define:
\begin{align*}
D:=dim_{ \mathds{k}}e_{0}Ae_{\omega}&=a+l+n-\#\{\text{linearly independent non monomial relations}\}.
\end{align*}

Given a finite set of $s$ equations generating the non monomial relations and having fixed an order of the branches, 
let $C=(a_{ij})\in \mathds{k}^{s\times n}$ be the matrix whose rows are the coefficients of each of these equations.
Replacing the given relations by those obtained from the reduced row echelon form of the matrix gives of course the same algebra; 
from now on we will always suppose that this matrix is already reduced.
Every non monomial relation will be of the form:
\[\rho_{i}=\alpha^{(k_{i})}+\sum_{j>k_{i}}b_{ij}\alpha^{(j)}.\]
We will call $W_{\rho_{i}}=\alpha^{(k_{i})}$ and $f_{\rho_{i}}=-\sum_{j>k_{i}}b_{ij}\alpha^{(j)}$, keeping in mind the idea that 
the word $W_{\rho_{i}}$ will be replaced in $A$ by $f_{\rho_{i}}$.
Let $\mathcal{R}$ be a minimal set of generators of $I$ containing $\rho_{i}$ for all $i$.
The set $\mathcal{R}$ will be the disjoint union of 
 $\mathcal{R}_{mon}$, consisting of monomial relations  and $\mathcal{R}_{nomon}=\{\rho_{i}\}$, the set of non monomial relations.

Since we are going to compute the Hochschild cohomology spaces of $A$, the first thing we need is a useful projective resolution of 
$A$ as $A$-bimodule; if possible a minimal one.


\subsection{The resolution}
We will next recall the definition of $n$-ambiguity from \cite{Sk}.


\begin{definition} Given $n\geq2$,
 \begin{enumerate}
  \item the path $p \in Q$ is a {\em left} $n$-{\em ambiguity} if there exist $u_0\in Q_1$ and 
   $u_1,\dots,u_n$ paths not in $I$ such that
    \begin{enumerate}[(i)]
      \item $p=u_0u_1\cdots u_n$,
      \label{def:nambiguedadiz1}
      \item for all $i$, and for any proper left divisor $d$ of $u_{i+1}$, the path $u_iu_{i+1}$ belongs to $I$ but 
      $u_id$ does not belong to $I$.
      \label{def:nambiguedadiz2}
    \end{enumerate}
   \item the path $p\in Q$ is a {\em right} $n$-{\em ambiguity} if there exist $v_0\in Q_1$ and
   $v_1,\dots,v_n$ paths not in $I$ such that
    \begin{enumerate}[(i)]
     \item $p=v_n\cdots v_0$,
     \item for all $i$, and for any proper right divisor $d$ of $v_{i+1}$, the path $v_{i+1}v_i$ belongs to  $I$ but $dv_i$ does not belong to $I$.
    \end{enumerate}
 \end{enumerate}
\end{definition}


As we have already mentioned, the subalgebra $E= \mathds{k}Q_{0}$ of $A$ is separable over $\mathds{k}$, so we can compute 
the resolution
relative to $E$. This resolution will come from a monomial order, see \cite{CS}. Let us fix an order
in $Q_{0}\cup Q_{1}$. For this, let us draw the toupie quiver $Q$ as follows

\medskip $$\xymatrix@R=6pt@C=6pt{  & & \ar[lld] \ar[ld]\bullet\save[]+<0pt,8pt>*{0}\restore  \ar[dr]\restore  \ar@/^4pc/[ddddd] \ar@/^6pc/[ddddd]&  \\
 \!\!\!\!\!\!\!\!\!\! _{e^{(1)}_1} \bullet \ar[d] & \!\!\!\!\!\!\!\!\!\!\!_{e^{(2)}_1} \bullet \ar[d] & \cdots  & \bullet\ar[d]& \\ 
  \vdots & \vdots &\cdots &\vdots& \,\,\,\,\,\,\,\,\,\,\,\,\,\, \cdots \\ \ar[d] & \ar[d] & & \ar[d]\\\!\!\!\!\! _{e^{(1)}_{n_1}} \bullet 
\ar@/^-1pc/[drr] &  \!\!\!\!\!_{e^{(2)}_{n_2}} \bullet \ar[dr] & \cdots & \bullet\ar[ld] & \\& & \bullet\save[]+<0pt,-8pt>*{\omega}\restore & & }$$      \vspace{.1 cm}\\

\vspace{.1 cm}


and for each $i$ denote 
$\{e_{j}^{(i)}/
1\leq j\leq n_{i}\}$ the set of vertices in the 
branch $\alpha^{(i)}$
such that
$e_{j}^{(i)}\neq 0$ and $e_{j}^{(i)}\neq \omega$.
Fix $\omega=min\ Q_{0}$, $0=max\ Q_{0}$ and $e_{j}^{(i)}<e_{k}^{(l)}$ if $l<i$ or if $i=l$ and $k<j$. For the arrows, 
fix $\alpha_{j}^{(i)}<\alpha_{k}^{(l)}$ if $l<i$ or $i=l$ and $k<j$. Let $\leq$ be the order in $\mathds{k}Q/I$ which is compatible 
with concatenation and
extends $<$. Choose $S=Mintip\ (I)$ as in \cite[Def. 2.8]{CS} for example. It is then known that $I$ is the two sided
ideal generated by $\{s-f_{s}\}_{s\in S}$ and that the reduction system $\mathfrak{R}=\{(s, f_{s})\}_{s\in S}$ is such that
every path is reduction unique, see for example \cite[Lemmas 2.4 and 2.10]{CS}. Explicitly, 
$\mathfrak{R}=\mathfrak{R}_{1}\cup \mathfrak{R}_{2}$,
where $\mathfrak{R}_{1}=\{(\sigma_{i},0)/ \text{$\sigma_{i}$ a monomial relation}\}$ and 
$\mathfrak{R}_{2}=\{(\alpha^{(k_{i})}, \Sigma_{j< k_{i}}b_{ij}\alpha^{(j)})/ a+l+m+1\leq k_{i}\leq a+l+m+n\}$.

In our situation, it is easy to see that all possible ambiguities arise from $\mathfrak{R}_{1}$, and since they correspond to the 
monomial part, they reduce to $0$. We denote by $\mathcal{B}$ the basis of $A$ obtained from the classes in $\mathds{k}Q/I$ of the 
irreducible paths.
From now on we will fix the reduction system. We are now ready to construct the resolution.

\begin{Notation}
 Given a path $u\in Q$, whenever we write $\sum u^{(1)}\otimes u^{(2)}\otimes u^{(3)}\in A\otimes_{E}A\otimes_{E}A$ we are considering 
 the sum of all the possible factorizations of $u$ as product $u^{(1)}u^{(2)}u^{(3)}=u$. If $u^{(2)}$ is an arrow we write 
 $\sum u^{(1)}\otimes \overleftarrow{u^{(2)}}\otimes u^{(3)}$. If $\rho_{i}$ is a non monomial relation, we know that it is of the 
 form $\alpha^{(k_{i})}+\sum_{j> k_{i}} b_{ij}\alpha^{(j)}$, then when we write 
 $\sum [\rho_{i}]^{(1)}\otimes [\overleftarrow{\rho_{i}}]^{(2)}\otimes [\rho_{i}]^{(3)}$ we refer to 
 $\sum [\alpha^{(k_{i})}]^{(1)}\otimes \overleftarrow{[\alpha^{(k_{i})}]^{(2)}}\otimes [\alpha^{(k_{i})}]^{(3)}+
 \sum_{j> k_{i}}b_{ij}\sum [\alpha^{(j)}]^{(1)}\otimes\overleftarrow{ [\alpha^{(j)}]^{(2)}}\otimes [\alpha^{(j)}]^{(3)}$.
\end{Notation}

In low degrees, we already know that the extended minimal resolution is
\[A\otimes_{E} \mathds{k}\mathcal{R}\otimes_{E}A\stackrel{d_{1}}{\rightarrow}A\otimes_{E} \mathds{k}Q_{1}\otimes_{E}A\stackrel{d_{0}}{\rightarrow}A\otimes_{E}A\stackrel{\mu}{\rightarrow}A\rightarrow 0\]

As usual, $\mu$, $d_{0}$ and $d_{1}$ are the $A$-bimodule maps defined by;
\begin{align*}
\mu(a\otimes b)&=ab~~for~~a,b\in A,\\
d_{0}(1\otimes v\otimes 1)&=v\otimes 1-1\otimes v~~ for~~ v\in Q_{1},\\
d_{1}(1\otimes \sigma\otimes 1)&=\Sigma \sigma^{(1)}\otimes \overleftarrow{\sigma^{(2)}}\otimes\sigma^{(3)} ~~for~~\sigma\in \mathcal{R}.
\end{align*}
For $n\geq 2$, let us
denote $P_{n}=A\otimes_{E} \mathds{k}\mathcal{A}_{n}\otimes_{E}A$, where $\mathcal{A}_{n}$ is the set of $n$-ambiguities.

Given $v=v_{0}\cdots v_{n}=w_{n}\cdots w_{0}\in \mathcal{A}_{n}$, the differential $d_n:P_n\to P_{n-1}$ will be defined 
--using the same notation as before-- as follows:
 \begin{equation*}
     d_{n}(1\otimes v\otimes 1)=\left\lbrace
  \begin{array}{cc}

     \sum v^{(1)}\otimes v^{(2)}\otimes v ^{(3)},\quad\text{with } v^{(2)}\in \mathcal{A}_{n-1}&\text{ if                                                                   $n$ is odd,} \\

     w_{n}\otimes w_{n-1}\dots w_{0}\otimes 1-1\otimes u_{0}\dots u_{n-1}\otimes u_{n}, &\text{if $n$ is even.}\\

  \end{array}
  \right.
\end{equation*}

Let us call $C_{min}(A)$ the sequence
\[\dots\longrightarrow P_{n}\stackrel{d_{n}}{\longrightarrow}P_{n-1}\cdots
\longrightarrow P_{2}\stackrel{d_{2}}{\longrightarrow} A\otimes_{E} \mathds{k}\mathcal{R}\otimes_{E}A \stackrel{d_{1}}{\longrightarrow}A\otimes_{E} \mathds{k}Q_{1}\otimes_{E}A \stackrel{d_{0}}{\longrightarrow}A\otimes_{E}A \longrightarrow 0\]

\begin{lemma}
The sequence $C_{min}(A)$ is a complex of projective $A$-bimodules.
\end{lemma}
\begin{proof}
For $n>2$, the only relations involved in $P_{n}$ are monomial and the proof that $d_{n}\circ d_{n+1}=0$ is in \cite{Sk}.
We have to verify that $d_{1}\circ d_{2}=0$ and $d_{2}\circ d_{3}=0$.
Let us first prove that $d_{1}\circ d_{2}=0$. Given a $2$-ambiguity $v=v_{0}v_{1}v_{2}=w_{2}w_{1}w_{0}$,

\begin{align*}
d_{1}\circ d_{2}(1\otimes v\otimes 1)&=d_{1}(w_{2}\otimes w_{1}w_{0}\otimes 1-1\otimes v_{0}v_{1}\otimes v_{2})\\&=\Sigma w_{2}(w_{1}w_{0})^{(1)}\otimes \overleftarrow{(w_{1}w_{0})^{(2)}}\otimes (w_{1}w_{0})^{(3)}\\
& - \Sigma (v_{0}v_{1})^{(1)}\otimes \overleftarrow{(v_{0}v_{1})^{(2)}} \otimes (v_{0}v_{1})^{(3)}v_{2}.
\end{align*}
In case there is a term in the first summand not appearing in the second one; then $(w_{1}w_{0})^{(3)}$ must be a proper divisor of
$v_{2}$,  which implies that $\overleftarrow{(w_{1}w_{0})^{(2)}}(w_{1}w_{0})^{(3)}$ is a divisor of $v_{2}$ and so
$w_{2}(w_{1}w_{0})^{(1)}$ would be divisible by $v_{0}v_{1}$, which is $0$ in $A$.
A similar argument is used to prove that every non zero term that appears in the second sum appears in the first one too.
\bigbreak
Let us now prove that $d_{2}\circ d_{3}=0$.
Given a $3$-ambiguity $v=u_{0}u_{1}u_{2}u_{3}=w_{3}w_{2}w_{1}w_{0}$,
$$d_{2}\circ d_{3}(v)=d_{2}(\Sigma_{v}v^{(1)}\otimes v^{(2)}\otimes v^{(3)})$$
with $v^{(2)}\in \mathcal{A}_{2}$.
The set of $2$-ambiguities appearing in the decomposition of $v$ is finite, so we can order the summands by length of the first term
$v^{(1)}$. Notice that two different terms cannot have $v^{(1)}$'s of the same length: if this happens, then one of the $v^{(2)}$'s 
would be
strictly shorter that the other, and this contradicts Lemma 3.1.5 of \cite{arten}.

Let us order the summands as
\[v_{1}^{(1)}\otimes v_{1}^{(2)}\otimes v_{1}^{(3)}< \cdots < v_{n}^{(1)}\otimes v_{n}^{(2)}\otimes v_{n}^{(3)}\]
with $|v_{i}^{(1)}|<|v_{i+1}^{(1)}|$ for $i=1,...,n-1$.

Given $t$, $i\leq t\leq n$, we know after the second item of Lemma 3.1.6 of \cite{arten} that $v_{t}^{(2)}$ and $v_{t+1}^{(2)}$ both 
have only two divisors in $\mathcal{R}$, respectively $v'_{t}$, $v''_{t}$ and $v'_{t+1}$, $v''_{t+1}$. The first part of Lemma 3.1.6 
implies that $v'_{t+1}$ equals $v''_{t}$, because if not $v^{(2)}_{t}$ and $v^{(2)}_{t+1}$ would share a $2$-ambiguity, which is 
impossible.

Using the notation $v_{t}^{(2)}=v'_{t}u_{2,t}=w_{2,t}v'_{t+1}$,
and observing that $u_{2,1}=u_{2}$, since $v'_{1}=u_{0}u_{1}$, and that $w_{2,n}=w_{2}$ since $v'_{n+1}=w_{1}w_{0}$, we get
\begin{align*}
d_{2}\circ d_{3}(1\otimes v\otimes 1)&=-\Sigma_{t=1}^{n}v_{t}^{(1)}w_{2,t}\otimes v'_{t+1}\otimes v_{t}^{(3)}+\Sigma_{i=1}^{n}v_{t}^{(1)}\otimes v'_{t}\otimes u_{2,t} v_{t}^{(3)}\\&=-w_{3}w_{2}\otimes v'_{n+1}\otimes 1+1\otimes v'_{1}\otimes u_{2}u_{3}\\&=0.
\end{align*}
Notice that, since $E$ is separable, $E\otimes_{\mathds{k}}E^{op}$ is semisimple, thus, all the $E$-bimodules appearing in the
complex are projective, implying that the $A$-bimodules $P_n$ are projective.
\end{proof}



\begin{proposition}\label{complex}
The complex $C_{min}(A)$ is a projective resolution of $A$ as $A^{e}$-module.
\end{proposition}
\begin{proof}
%
%
The only thing left to prove is the exactness of the complex, and this follows from \cite[Theorem 4.1]{CS}.
Nevertheless, we shall construct a contracting homotopy $\{s_{i}\}_{i\geq -1}$ that will be useful in Section \ref{comp}.
 For $i=-1$ and $a\in A$, let
$$s_{-1}(a)=a\otimes e_{t(a)}.$$

For $i\geq 0$, let $a$ be an element of the $\mathds{k}$-basis of irreducible paths $\mathcal{B}$. It is sufficient to define $s_{i}$ 
on the elements $e_{s(w)}\otimes w\otimes a\in P_{i}$ and extend it to a morphism of left $A$-modules.
Set
$$ s_{0}(e_{s(a)}\otimes a)=  -\sum a^{(1)}\otimes \overleftarrow{a^{(2)}}\otimes a^{(3)},$$

\begin{equation*}
s_{1}(e_{s(\alpha)}\otimes \alpha\otimes a)=\left\{
                                              \begin{array}{ll}
                                                e_{0}\otimes \rho_{i}\otimes e_{\omega} , & \hbox{if  $\alpha a=W_{\rho_{i}}$,} \\
                                                e_{s(\alpha)}\otimes \sigma_{j}\otimes b, & \hbox{if $\alpha a=\sigma_{j}b$ in $Q$ with $\sigma_{j}$ monomial,} \\
                                                0, & \hbox{otherwise.}
                                              \end{array}
                                            \right.
\end{equation*}

For $i= 2$, given $r\in \mathcal{R}$, set $$s_{2}(e_{s(r)}\otimes r \otimes a)=-\sum (ra)^{(1)}\otimes (ra)^{(2)}\otimes (ra)^{(3)}$$
 with $(ra)^{(2)}\in \mathcal{A}_{2}$ if $ra$ contains a $2$-ambiguity and zero otherwise.
Observe that if $r=\rho_{i}$ is a non monomial relation, then $s_{2}(e_{s(r)}\otimes r\otimes a)=0$.

For $i> 2$, given $w\in \mathcal{A}_{i-1}$, $$s_{i}(e_{s(w)}\otimes w \otimes a)=(-1)^{i+1}\sum (wa)^{(1)}\otimes (wa)^{(2)}\otimes (wa)^{(3)}$$
 with $(wa)^{(2)}\in \mathcal{A}_{i}$ if $wa$ contains an $i$-ambiguity, and zero otherwise.
 It is straightforward to verify that $\{s_{i}\}_{i\geq -1}$ is a contracting homotopy.
 \end{proof}

\begin{remark}
To check that the projective resolution is minimal consider $r^{e}$, the Jacobson radical of $A^{e}$. This radical is generated by the elements of the form $\alpha\otimes e_{j}$ and $e_{j}\otimes\beta$ with  $i,j \in Q_{0}$ and $\alpha, \beta\in Q_{1}$ since $rad(A\otimes A^{op})=A\otimes rad(A^{op})+rad(A)\otimes A^{op} $.
It is easy to see that $Im(d_{i})\subset P_{i-1}r^{e}$, so we conclude that $C_{min}(A)$ is minimal.
\end{remark}

\section{Computation of Hochschild Cohomology}\label{calculocohomo}

In this section we will construct an explicit basis for each cohomology  space of $A$. The knowledge of such bases will be
useful for the computation of the deformations of toupie algebras and for description of the Gerstenhaber structure.
Applying the functor $\Hom_{A^{e}}(-,A)$ to the minimal resolution $C_{min}(A)$
and using again the canonical isomorphism $
F_{X}:\Hom_{A^{e}}(A\otimes_{E} X\otimes_{E} A, A)\rightarrow \Hom_{E^{e}}(X,A)$, that we will simply denote $F$,
we obtain the following complex:
\[0\longrightarrow \Hom_{E^{e}}(E,A)\stackrel{D_{0}}{\longrightarrow} \Hom_{E^{e}}( \mathds{k}Q_{1},A)\stackrel{D_{1}}{\longrightarrow}\Hom_{E^{e}}( \mathds{k}\mathcal{R},A)\stackrel{D_{2}}{\longrightarrow}\Hom_{E^{e}}( \mathds{k}\mathcal{A}_{2},A) \longrightarrow \dots\]
where $D_{i}=F\circ \Hom_{A^{e}}(d_{i}, A)\circ F^{-1}$ for all $i$.

For any $E$-bimodule $W$ and any $f\in \Hom_{E^{e}}(W,A)$, the equalities $e_{i}f(w)e_{j}=f(e_{i}we_{j})$,
for every $i,j\in Q_{0}$ and $w\in W$ mean that $w$ and $f(w)$ share source and target. Using Cibils's notation in \cite{Cib}, 
we will write $w\|f(w)$ for the morphism
in $\Hom_{E^{e}}(W,A)$ sending a basis element $w$ of $W$ to $f(w)$ and the other basis elements to zero. Also, given sets 
$H$ and $G$, we will
denote $ \mathds{k}(H\| G)$ the $\mathds{k}$-span of $\{h\| g\}_{h\in H, g\in G}$.

Let us denote ${}_{0}\mathcal{B}_{\omega}$ the subset of all branches belonging to $\mathcal{B}$.

We will next describe explicitly the vector spaces and the differentials appearing in the complex.

\medskip

Since toupie algebras have no cycles, $\Hom_{E^{e}}(E,A)$ is just $ \mathds{k}(Q_{0}\|Q_{0})$. 
A careful look shows that the spaces appearing in degrees $1$ are $2$ are respectively
\[\Hom_{E^{e}}( \mathds{k}Q_{1},A)= \mathds{k}(Q_{1}\| Q_{1})+ \mathbb{ \mathds{k}}(Z\| ({}_{0}\mathcal{B}_{\omega}-Z)).\]
\[\Hom_{E^{e}}( \mathds{k}\mathcal{R},A)= \mathds{k}(\mathcal{R}_{nomon}\|{}_{0}\mathcal{B}_{\omega})+ \mathds{k}(_{0}(\mathcal{R}_{mon})_{\omega}\|{}_{0}\mathcal{B}_{\omega}).\]
For $i\geq 2$,
\[\Hom_{E^{e}}( \mathds{k}\mathcal{A}_{i},A)= \mathds{k}({}_{0}(\mathcal{A}_{i})_{\omega}\|{}_{0}\mathcal{B}_{\omega}).\]

The differentials are as follows. Given $j\in Q_{0}$ and $e_{j}\| e_{j}\in \Hom_{E^{e}}(E, A)$,
$$D_{0}(e_{j}\| e_{j})=\sum_{\alpha: t(\alpha)=j}\alpha\| \alpha-\sum_{\beta: s(\beta)=j}\beta\|\beta.$$
Given $\alpha_{r}\in Q_{1}$,
\begin{equation*}
     D_{1}(\alpha_{r}\| \alpha_{r})=\left\lbrace
  \begin{array}{l}
  \sum_{i:\alpha^{(j)}\in \rho_{i}} \rho_{i}\| b_{ij}\alpha^{(j)}, ~~  \text{if } ~~ \alpha_{r} ~~\text{is}~~\text{in}~~ \alpha^{(j)}\text{ and $\alpha^{(j)}$} \text{ is a branch of $\rho_{i}$},  \\
     0,~~ \qquad\qquad\qquad
     \qquad \text{otherwise.}
  \end{array}
  \right.
\end{equation*}
For $\alpha^{(h)}\| \alpha^{(j)}\in Z\| ({}_{0}\mathcal{B}_{\omega}-Z)$,
\begin{equation*}
  D_{1}(\alpha^{(h)}\| \alpha^{(j)})=0.
\end{equation*}
Finally, $D_{i}=0$ for $i\geq 2$.\bigbreak

An easy verification shows that, as it is well known when the quiver contains no oriented cycles, $\HH^{0}(A)=\langle\sum_{i\in Q_{0}}e_{i}\| e_{i}\rangle$.

\subsection{Computation of $\HH^{1}(A)$}\label{caluloH1}
To compute $\HH^{1}(A)$ we first obtain a basis of $Ker D_{1}$ that we modify afterwards so that it contains a basis of 
$Im D_{0}$.
For that purpose, using the branches of the quiver of length greater than $1$ that do not contain monomial relations, we will 
construct a non oriented  graph $Q_{\rho}$ as follows.
\begin{definition}\label{qro}
The graph $Q_{\rho}$ is the following:
\begin{enumerate}
  \item its vertices are in bijection with the branches in $Q$ that are not arrows and do not contain monomial relations,
  labeled under the name of the corresponding branch,
  \item there is an edge between the vertices $\alpha^{(k)}$ and $\alpha^{(p)}$ if there exists a relation that involves both branches.
\end{enumerate}
We will denote $Q_{\rho}^{k}$, $1\leq k\leq r$, the connected components of $Q_{\rho}$.
\end{definition}
\begin{remark}\label{dimH1}
We recall the dimension of $\HH^{1}(A)$, computed in \cite{GL}: 
\[dim_{ \mathds{k}}\HH^{1}(A)=r+m+Da-1.\]
\end{remark}
Let us define the following four subsets of $ \mathds{k}(Q_{1}\|\mathcal{B})$:
\begin{enumerate}

\item $C_{1}=\{\alpha^{i}_{j}\| \alpha^{i}_{j}:\alpha^{(i)}\text{ is a branch which contains monomial relations}\}$.
\item $C_{2}=Z\|_{0}\mathcal{B}_{\omega}$.

\item $C_{3}=\{\alpha^{i}_{j}\| \alpha^{i}_{j}- \alpha^{i}_{0}\|\alpha^{i}_{0}:j\neq 0, |\alpha^{(i)}|>1 \text{ and $\alpha^{(i)}$ does not contain monomial relations}\}.$
\item $C_{4}=\{\sum_{\alpha^{(i)} \in Q_{\rho}^{k} } \alpha^{i}_{0}\| \alpha^{i}_{0}: k=1\dots, r\}$.
\end{enumerate}
\begin{lemma}
The set $U=\bigcup_{i=1}^{4}C_{i}$ is a basis of $Ker D_{1}$.
\end{lemma}
\begin{proof}
Since $dim_{ \mathds{k}}\HH^{1}(A)=r+m+Da-1$, our knowledge of the dimension of $\im(D_0)$ implies that $dim_{ \mathds{k}} Ker D_{1}=r+m+Da+\# Q_{0}-2$. By construction, $U$ is linearly independent. 
Observe that the disjoint union $C_{1}\cup C_{3}$ is in bijection with the set consisting of all the arrows that do not start 
in $0$ and the arrows of the monomial part that start in $0$, so $\# C_{1}+\# C_{3} =\# Q_{0}-2+m$. 
Since clearly $\# C_{2}=Da$ and $\# C_{4}=r$ we conclude that $\# U=r+m+Da+\# Q_{0}-2$.
\end{proof}

Next we will replace $U$ by another basis containing a basis of $Im D_{0}$; we will proceed as follows.
\begin{enumerate}
\item Replace each element $\alpha^{i}_{j}\| \alpha^{i}_{j}$ in $C_{1}$ with $j\neq 0$ by $\alpha^{i}_{j}\| \alpha^{i}_{j}-
\alpha^{i}_{0}\| \alpha^{i}_{0}$. 
Let us call $C_{1}'$ the set consisting of the modified elements and 
\[C''_{1}=\{\alpha_{0}^{i}\|\alpha_{0}^{i}: \alpha^{(i)} 
\hbox{ is a branch that contains monomial relations}\},\] its complement in $C_1$.
\item

    \begin{itemize}

    \item If $a>0$, first replace the elements in $C_{2}$ of the form $\alpha^{(i)}\|\alpha^{(i)}$ with $i\neq 1$ by 
    $\alpha^{(i)}\|\alpha^{(i)}-\alpha^{(1)}\|\alpha^{(1)}$ and then replace $\alpha^{(1)}\|\alpha^{(1)}$ by 
    $s=\sum_{\alpha^{(i)}} \alpha^{i}_{0}\| \alpha^{i}_{0}$ where the sum ranges over all the branches of the quiver.
    \item If $a=0$, let us call $\alpha^{(B)}$ the last branch of the toupie algebra and consider the element of $U$ in 
    $C_{1} \bigcup C_{4}$ that contains $\alpha_{0}^{B}\|\alpha_{0}^{B}$ as a summand. Replace this element by 
    $s=\sum_{\alpha^{(i)}} \alpha^{i}_{0}\| \alpha^{i}_{0}$ where the sum ranges over all the branches of the quiver.
     \end{itemize}
\end{enumerate}

Let us call $\widehat{U}$ the set obtained from $U$ in this way. After some direct and tedious computations, it turns out that the 
$\mathds{k}$-vector spaces generated by $U$ and $\widehat{U}$ coincide.

\begin{lemma}
The set $K=\{s\}\bigcup C'_{1}\bigcup C_{3}$ is a basis of $Im D_{0}$.
\end{lemma}
\begin{proof}
We know that $dim_{ \mathds{k}}Im D_{0}=\#Q_{0}-1$ because on one hand $dim_{ \mathds{k}}Ker D_{0}=1$ and on the other 
hand $dim_{ \mathds{k}}\Hom_{E^{e}}(E,A)=\#Q_{0}$. Since $K\subset Im D_{0}$, just observe that $\#C'_{1}+\#C_{3}=\#Q_{0}-2$ and $K$ is linearly independent.
\end{proof}
The proof of the following theorem is now immediate.
\begin{theorem}\label{baseH1}
The classes of the elements of $\widehat{U}-K$ in $Ker D_{1}/Im D_{0}$ form a basis of $\HH^{1}(A)$.
\end{theorem}

\subsection{Computation of $\HH^{i}(A)$ for higher degrees}
In order to obtain a basis of $\HH^{2}(A)$ we will follow the same lines as for $\HH^{1}(A)$. Recall that $D_{2}=0$ and 
so $Ker D_{2}=\Hom_{E^{e}}( \mathds{k}\mathcal{R},A)$.
\begin{remark}\label{remarkorder}
To compute a basis of $Im D_{1}$ we will need to define an order $\prec'$ in $\mathcal{R}_{nomon}$. Every element in 
$\mathcal{R}_{nomon}$ is of the form 
\[\rho_{i}=\alpha^{(k_{i})}+\sum_{j>k_{i}}b_{ij}\alpha^{(j)}\]
so we say that $\rho_{i}\prec' \rho_{j}$ if $k_{i}< k_{j}$. This order induces an order on the elements of 
$\mathcal{R}_{nomon}\|{}_{0}\mathcal{B}_{\omega}$ as follows,
 \[\rho_{i}\|\alpha^{(k)}\prec \rho_{j}\|\alpha^{(l)} \text{ if } \rho_{i}\prec' \rho_{j} \text{ or } \rho_{i}=\rho_{j} \text{ and } k<l.\]

For each connected component of $Q_{\rho}$ associated to a non monomial relation there exists a relation containing 
$\alpha^{(k_{i})}$ with $k_{i}$ maximum. We will call this relation the last one of the component.
\end{remark}

Denote by $X$ the set consisting of all the elements in ${}_{0}\mathcal{B}_{\omega}$ that are involved in non monomial relations and 
define $d=\#X$.\\

Consider the following subsets of $ \mathds{k}(\mathcal{R}\|_{0}\mathcal{B}_{\omega})$:

\begin{itemize}

\item $B_{1}=\{\rho_{i}\|f_{\rho_{i}}:\rho_{i}\in \mathcal{R}_{nomon} \hbox{ such that } \rho_{i}\hbox{ is not the last one of }Q^{k}_{\rho}, \hbox{ with }k=1,\dots ,r\}$,

\item $\displaystyle B_{2}=\{\sum_{i: \alpha^{(h)}\in \rho_{i}, h\neq k_{i}} \rho_{i}\| b_{ik}\alpha^{(h)}\ :\ \alpha^{(h)}\in X\}$.\\

\end{itemize}

Observe that $\# B_{1}=n-d+l-r$ and $\#B_{2}=d$.

\begin{lemma}
The set $B_{1}\cup B_{2}$ is a basis of $Im D_{1} $.
\end{lemma}
\begin{proof}
It is clear that $B_{1}\cup B_{2}\subset Im D_{1}$ since
$B_{1}\subset\{D_{1}(\alpha_{0}^{(i)}\|\alpha_{0}^{(i)}):
\alpha^{(i)}=W_{\rho_{j}} \hbox{ for some } j\}$ and
$B_{2}=\{D_{1}(\alpha_{0}^{(i)}\|\alpha_{0}^{(i)}):\alpha^{(i)}\in X\}$.
We know that $dim_{ \mathds{k}}Ker D_{1}=r +m + Da + \# Q_{0}-2$ and, on the other hand,
$dim_{ \mathds{k}}\Hom_{E^{e}}( \mathds{k}Q_{1},A)=\# Q_{1}+(D-1)a=m+l+n+Da+\# Q_{0}-2$.
From the previous computations we deduce that $dim_{ \mathds{k}} Im D_{1}=l+n-r$. 

Observe that the union $B_{1}\cup B_{2}$ is disjoint: an element in the intersection of $B_{1}$ and $B_{2}$ must be of the form 
$t=\rho_{i}\|c\alpha^{(k)}$ with $ \alpha^{(k)}\neq \alpha^{(k_{i})} $ the only branch except for $\alpha^{(k_{i})}$ in  $\rho_{i}$ 
and $c\in \mathds{k}$. Now, since $t$ has only one summand, the branch  $\alpha^{(k)}$ is not
involved in any other relation. On the other hand  $\alpha^{(k_{i})} $, by definition, is only involved in the relation $\rho_i$. 
Thus, the branches $ \alpha^{(k)}$ and $\alpha^{(k_{i})}= W_{\rho_{i}} $ are the only ones in their connected component of $Q_{\rho}$.
So $\rho_{i} = \alpha^{(k_{i})} - c\alpha^{(k)}$ is the unique relation  of $Q^{k}_{\rho}$, for some $k$, and this cannot happen 
due to the definition of $B_1$. 

Since $\# B_{1}=n-d+l-r$, $\#B_{2}=d$ and the union is disjoint, it follows that $\#(B_{1}\cup
B_{2})=\#B_{1}+\#B_{2}=l+n-r$, so it is sufficient to prove that $B_{1}\cup B_{2}$ generates
$Im(D_{1})$. By definition,
$D_{1}(\alpha_{0}^{(i)}\|\alpha_{0}^{(i)})$ coincides with $D_{1}(\alpha_{j}^{(i)}\|\alpha_{j}^{(i)})$
for any arrow $\alpha_{j}^{(i)}$ in the branch $\alpha^{(i)}$, also
$D_{1}(\sum_{i: \alpha^{(i)}\in Q_{\rho}^{k}}\alpha_{0}^{(i)}\|
\alpha_{0}^{(i)})= 0$, for every connected component $Q_{\rho}^{k}$ of $Q_{\rho}$, so the result is proven.
\end{proof}

Again, we will modify the basis $B_{1}\cup B_{2}$, following the next steps.
\begin{enumerate}
\item For every element $w$ in $B_{2}$ with a summand of the form $\rho_{j}\| b_{jk} \alpha^{(k)}$ where $\alpha^{(k)}$ is the first 
branch of $f_{\rho_{j}}$ and $\rho_{j}$ is not the last relation for any connected component of $Q_{\rho}$, subtract from $w$ the 
element $\rho_{j}\|f_{\rho_{j}}$ which belongs to $B_{1}$. Let us call $B'_{2}$ the set obtained from $B_{2}$ in this way. 
The vector spaces spanned by $B_{1}\cup B_{2}$ and by $B_{1}\cup B'_{2}$ are equal. 
\item Consider the order defined in \ref{remarkorder} for $\mathcal{R}_{nomon}\|_{0}\mathcal{B}_{\omega}$ and the 
matrix associated to the equations corresponding to the rows of $B'_{2}$ in its echelon form. 
Reconstruct the set in $\mathcal{R}_{nomon}\|_{0}\mathcal{B}_{\omega}$ from the echelon form and call it $B''_{2}$.
\end{enumerate}
Since $\langle B_{1}\cup B''_{2}\rangle=\langle B_{1}\cup B'_{2}\rangle =\langle B_{1}\cup B_{2}\rangle$, we conclude 
that $B_{1}\cup B''_{2}$ is a basis of $Im D_{1}$.\bigbreak

Now we will modify the set $$N=\{\rho_{i}\| \alpha^{(k)}: \rho_{i}\in \mathcal{R}_{nomon} \text{ and } \alpha^{(k)}\in {}_{0}\mathcal{B}_{\omega}\}\cup\{\sigma_{i}\| \alpha^{(k)}:\sigma_{i}\in {}_{0}(\mathcal{R}_{mon})_{\omega} \text{ and } \alpha^{(k)}\in {}_{0}\mathcal{B}_{\omega}\},$$ which is clearly a basis of $\Hom_{E^{e}}( \mathds{k}\mathcal{R},A)$, in order to obtain another basis $N'$ of $\Hom_{E^{e}}( \mathds{k}\mathcal{R},A)$ including $B_{1}\cup B''_{2}$.
The procedure is as follows.
\begin{enumerate}
\item To have the elements of $B_{1}$ in our basis, given $\rho_{i}\in \mathcal{R}_{nomon}$ and $\alpha^{(k)}\in f_{\rho_{i}}$ with 
$k$ the minimum of $f_{\rho_{i}}$, replace $\rho_{i}\| \alpha^{(k)}$ by $\rho_{i}\|f_{\rho_{i}}$.
\item To get the elements of $B''_{2}$ in our new basis consider the pivot of every element in $B''_{2}$ and replace the element 
having this pivot in $N$ by the corresponding one in $B''_{2}$.
\end{enumerate}

\begin{theorem}\label{teoremaH2}
The classes of $N'-B_{1}\cup B''_{2}$ in $\Hom_{E^{e}}( \mathds{k}\mathcal{R}, A)/Im(D_{1})$ form a basis of $\HH^{2}(A)$.
\end{theorem}

The last theorem of this section gives a basis of the Hochschild cohomology of $A$ for higher degrees.

\begin{theorem}\label{teoremaHi}
Given $i\geq 3$, the set $B_{i}=\{w_{j}\| \alpha^{(k)}:\alpha^{(k)}\in {}_{0}\mathcal{B}_{\omega}~~and~~w_{j}\in {}_{0}(\mathcal{A}_{i-1})_{\omega}\}$ is a basis of $\HH^{i}(A)$.
\end{theorem}
\begin{proof}
Since $D_{i}=0$ for $i\geq 3$ it follows that $\HH^{i}(A)=\Hom_{E^{e}}( \mathds{k}\mathcal{A}_{i-1},A)$. The set $B_{i}$ is a basis of 
this $\mathds{k}$-vector space so it is also a basis of $\HH^{i}(A)$.
\end{proof}


\section{Comparison morphisms}\label{comp}
For the computation of Hochschild cohomology we used a minimal resolution. 
One possible way of computing the Gerstenhaber structure is via comparison morphisms between this minimal resolution and 
the $E$-reduced Bar resolution in whose existence is guaranteed since both complexes are projective resolutions of 
$A$ as $A^{e}$-module.
We will describe them explicitly in the sequel.
Let us call $\varphi:C_{min}(A)\to C_{BarE}(A)$ and $\eta:C_{BarE}(A)\to C_{min}(A)$ a pair of comparison morphisms between both resolutions:\\

{\small$
  \begin{array}{cccccccccccc}
  \dots A\otimes_{E} \mathds{k}\mathcal{A}_{2}\otimes_{E}A&\stackrel{d_{2}}{\longrightarrow}&  A\otimes_{E}\mathds{k}\mathcal{R}\otimes_{E}A &  \stackrel{d_{1}}{\longrightarrow} & A\otimes_{E}\mathds{k}Q_{1}\otimes_{E}A & \stackrel{d_{0}}{\longrightarrow} &A\otimes_{E}A   & \rightarrow  0& & \\
      _{\varphi_{3}}\downarrow\uparrow_{\eta_{3}}& & _{\varphi_{2}}\downarrow \uparrow_{\eta_{2}} & &  _{\varphi_{1}}\downarrow \uparrow_{\eta_{1}} & & _{\varphi_{0}}\downarrow \uparrow_{\eta_{0}}    &   &  &\\

 \dots A\otimes_{E}\overline{A}^{\otimes 3}\otimes_{E}A & \stackrel{\delta_{2}}{\longrightarrow} &A\otimes_{E}\overline{A}^{\otimes 2}\otimes_{E} A & \stackrel{\delta_{1}}{\longrightarrow} & A\otimes_{E}\overline{A}\otimes_{E}A& \stackrel{\delta_{0}}{\longrightarrow} &A\otimes_{E} A  &
   \rightarrow 0.  &  &  \\
  \end{array}$}
\vspace{1em}

After applying the functor $\Hom_{A^{e}}(-,A)$ and considering the natural identifications at the beginning of Section 3  we obtain 
the maps $\varphi^{\ast}$ and $\eta^{\ast}$ and the following diagram:
\\

{\footnotesize$
  \begin{array}{cccccccccccc}
    0 \longrightarrow &\Hom_{E^{e}} (E,A) & \stackrel{D_{0}}{\longrightarrow} & \Hom_{E^{e}}( \mathds{k}Q_{1},A) & \stackrel{D_{1}}{\longrightarrow} & \Hom_{E^{e}}( \mathds{k}\mathcal{R}, A )& \stackrel{D_{2}}{\longrightarrow} & \Hom_{E^{e}}( \mathds{k}\mathcal{A}_{2},A)\dots &&  \\
      & _{\eta_{0}^{\ast}}\downarrow\uparrow_{\varphi_{0}^{\ast}} &  & _{\eta_{1}^{\ast}}\downarrow \uparrow_{\varphi_{1}^{\ast}} &  & _{\eta_{2}^{\ast}}\downarrow \uparrow_{\varphi_{2}^{\ast}} &  & _{\eta_{3}^{\ast}}\downarrow\uparrow_{\varphi_{3}^{\ast}} & & &\\
    0 \longrightarrow&\Hom_{E^{e}} (E,A) &\stackrel{\delta_{0}^{\ast}}{\longrightarrow} & \Hom_{E^{e}}(\overline{A},A) & \stackrel{\delta_{1}^{\ast}}{\longrightarrow} & \Hom_{E^{e}}(\overline{A}\otimes_{E}\overline{A},A) & \stackrel{\delta_{2}^{\ast}}{\longrightarrow} &
     \Hom_{E^{e}}(\overline{A}^{\otimes_{E} 3},A)\dots&& &  \\
  \end{array}
$}
\vspace{1em}

The maps $\varphi^{*}$ and $\eta^{*}$ induce isomorphisms at the cohomology level that we will still denote $\varphi^{*}$ and $\eta^{*}$.
We start the description of the morphisms of complexes $\varphi:C_{min}(A)\to C_{BarE}(A)$ and $\eta:C_{BarE}(A)\to C_{min}(A)$.
    Firstly, $\varphi_{0}$ and $\eta_{0}$ can both be chosen as the identity of $A\otimes_{E}A$. We will construct $\varphi_{i}$ for $i>0$ using the homotopy $t_{\ast}$ defined in Lemma \ref{Ered}.

Given $i\in \mathbb{N}$ and $1\otimes_{E}u\otimes_{E}1\in P_{i-1}$, we define:
$$\varphi_{i}(1\otimes_{E}u\otimes_{E}1)=t_{i-1}\varphi_{i-1}d_{i-1}(1\otimes_{E}u\otimes_{E}1).$$

and the we extend $\varphi_{i}$ as an $A$-bimodule morphism. Given $\lambda, \mu\in\mathcal{B}$:
$$\varphi_{i}(\lambda\otimes_{E}u\otimes_{E}\mu)=\lambda t_{i-1}\varphi_{i-1}d_{i-1}(1\otimes_{E}u\otimes_{E}1)\mu.$$
Finally we extend it linearly.

\begin{proposition}\label{phiesdecomplejos}
The family of functions $\varphi=\{\varphi_{i}\}_{i\in \mathbb{N}\cup \{0\}}$ is a morphism of complexes.
\end{proposition}
\begin{proof}
To prove that $\varphi$ is a morphism of complexes we need to verify that the corresponding diagrams commute, that means 
that $\delta_{i}\varphi_{i+1}=\varphi_{i}d_{i}$ for all $i$. We will prove it by induction. 
The case $i=0$ is clear. For the inductive step, let us suppose that  $\delta_{j-1}\varphi_{j}=\varphi_{j-1}d_{j-1}$ for 
any $j> 0$ and prove that $\delta_{j}\varphi_{j+1}=\varphi_{j}d_{j}$. Firstly, observe that it is sufficient to 
prove the result for $1\otimes_{E}u\otimes_{E}1\in P_{j}$ since $\varphi_{i}, d_{i}$ and $\delta_{i}$ are morphisms of 
$A$-bimodules for all $i$.

By definition of $\varphi_{j+1}$,
\begin{align*}
\delta_{j}\varphi_{j+1}(1\otimes_{E}u\otimes_{E}1)&=\delta_{j}t_{j}\varphi_{j}d_{j}(1\otimes_{E}u\otimes_{E}1)\\&=
(Id-t_{j-1}\delta_{j-1})\varphi_{j}d_{j}(1\otimes_{E}u\otimes_{E}1).
\end{align*}


The inductive hypothesis says that $\delta_{j-1}\varphi_{j}=\varphi_{j-1}d_{j-1}$, as a consequence
 \[t_{j-1}\delta_{j-1}\varphi_{j}d_{j}(1\otimes_{E}u\otimes_{E}1)=t_{j-1}\varphi_{j-1}d_{j-1}d_{j}(1\otimes_{E}u\otimes_{E}1)\]
 which is zero since $d_{j-1}d_{j}=0$, and we obtain the desired equality.
\end{proof}

\begin{remark}
In lower degrees, using the inductive definition, we obtain the explicit formulas:
\[\varphi_{1}(\lambda\otimes_{E}\alpha\otimes_{E}\mu)=\lambda\otimes_{E}\alpha\otimes_{E}\mu.\]
\[\varphi_{2}(\lambda\otimes_{E}r\otimes_{E}\mu)=\sum \lambda\otimes_{E}\overline{r^{(1)}}\otimes_{E}\overleftarrow{r^{(2)}}\otimes_{E}r^{(3)}\mu.\]
\[\varphi_{3}(\lambda\otimes_{E} u\otimes_{E}\mu)=\sum \lambda\otimes_{E}\overline{w_{2}}\otimes_{E}(\overline{w_{1}w_{0})^{(1)}}\otimes_{E}\overleftarrow{(w_{1}w_{0})^{(2)}}\otimes_{E}(w_{1}w_{0})^{(3)}\mu\]
where $u=w_{2}w_{1}w_{0}$ as a right $2$-ambiguity.
\end{remark}

For the definition of $\{\eta_{i}\}_{i\in \mathbb{N}}$ we will use a similar procedure, this time with
the homotopy $\{s_{i}\}_{i\geq -1}$ constructed in Proposition \ref{complex} for our minimal resolution.
Namely, given $i\in \mathbb{N}$ and 
$1\otimes_{E}\overline{a_{1}}\otimes_{E}\dots\otimes_{E}\overline{a_{i}}\otimes_{E} 1\in A\otimes_{E}\overline{A}^{\otimes_{E} i}\otimes_{E}A$:
\[\eta_{i}(1\otimes_{E}\overline{a_{1}}\otimes_{E}\dots \otimes_{E}\overline{a_{i}}\otimes_{E} 1)=s_{i-1}\eta_{i-1}\delta_{i-1}(1\otimes_{E}\overline{a_{1}}\otimes_{E}\dots \otimes_{E}\overline{a_{i}}\otimes_{E} 1),\]
and we extend $\eta_{i}$ as a morphism of $A$-bimodules and then linearly.

The proof of the following proposition is analogous to  the previous one and will be omitted. 

\begin{proposition}
The family of functions $\eta=\{\eta_{i}\}_{i\in \mathbb{N}}$ is a morphism of complexes.
\end{proposition}

\bigskip

Let us obtain the explicit formulas of $\eta_{i}$ for lower degrees.
For $i=1$, and $1\otimes_{E}\overline{a}\otimes_{E}1 \in A\otimes_{E}\overline{A}\otimes_{E}A$:
\begin{align*}
\eta_{1}(1\otimes_{E}\overline{a}\otimes_{E}1)&=s_{0}\eta_{0}(a\otimes_{E}1-1\otimes_{E}a)\\&=
s_{0}(a\otimes_{E}1-1\otimes_{E}a)\\&=\sum a^{(1)}\otimes_{E}\overleftarrow{a^{(2)}}\otimes_{E}a^{(3)},
\end{align*}
since $s_{0}(a\otimes_{E}1)=0$.

We continue now with $\eta_{2}$. Given $1\otimes_{E}a_{1}\otimes_{E}a_{2}\otimes_{E}1\in A\otimes_{E}\overline{A}^{\otimes_{E} 2}\otimes_{E}A$ with $a_{1}, a_{2}\in \overline{\mathcal{B}}=\mathcal{B}-\{e_{x}:x\in Q_{0}\}$,
\begin{align*}
\eta_{2}(1\otimes_{E}a_{1}\otimes_{E}a_{2}\otimes_{E}1)&=s_{1}\eta_{1}\delta_{1}(1\otimes_{E}a_{1}\otimes_{E}a_{2}\otimes_{E}1)\\&=
s_{1}\eta_{1}(a_{1}\otimes_{E}a_{2}\otimes_{E}1-1\otimes_{E}a_{1}a_{2}\otimes_{E}1+1\otimes_{E}a_{1}\otimes_{E}a_{2}).
\end{align*}

There are three different cases to consider:
\begin{enumerate}
\item If $a_{1}a_{2}\in \mathcal{B}$, then
\begin{align*}
\eta_{1}\delta_{1}(1\otimes_{E}a_{1}\otimes_{E}a_{2}\otimes_{E}1)&=\sum a_{1}a_{2}^{(1)}\otimes_{E}\overleftarrow{a_{2}^{(2)}}\otimes_{E}a_{2}^{(3)}\\&-\sum (a_{1}a_{2})^{(1)}\otimes_{E}\overleftarrow{(a_{1}a_{2})^{(2)}}\otimes_{E}(a_{1}a_{2})^{(3)}\\&+\sum a_{1}^{(1)}\otimes_{E}\overleftarrow{a_{1}^{(2)}}\otimes_{E}a_{1}^{(3)}a_{2}=0.\\
\end{align*}

So, if $a_{1}a_{2}\in \mathcal{B}$: 
$\eta_{2}(1\otimes_{E}a_{1}\otimes_{E}a_{2}\otimes_{E}1)=0$.

\item If $a_{1}a_{2}$ contains at least one monomial relation, let $\{\sigma_{1}, \sigma_{2},\dots ,\sigma_{p}\}$ be the set of 
monomial relations that appear in $a_{1}a_{2}$, ordered by their starting point in $a_{1}$. For every $\sigma_{i}$ we write 
$a_{1}a_{2}=(\sigma_{i})^{l}\sigma_{i}(\sigma_{i})^{r}$ where $(\sigma_{i})^{l}$ is the part of $a_{1}a_{2}$ on the left of 
$\sigma_{i}$ and $(\sigma_{i})^{r}$ is the right part.
    By definition of  $\eta_{1}$,
\noindent\begin{align*}
\noindent\eta_{2}(1\otimes_{E}a_{1}\otimes_{E}a_{2}\otimes_{E}1)&=\sum s_{1}(a_{1}a_{2}^{(1)}\otimes_{E}\overleftarrow{a_{2}^{(2)}}\otimes_{E}a_{2}^{(3)})\\&+\sum s_{1}(a_{1}^{(1)}\otimes_{E}\overleftarrow{a_{1}^{(2)}}\otimes_{E}a_{1}^{(3)}a_{2}).
\end{align*}

The first summand is zero since $a_{2}\in \overline{\mathcal{B}}$ and this implies that no monomial relation in $a_{1}a_{2}$ starts 
after the end of $a_{1}$. With respect to the second summand, observe that if $\overleftarrow{a_{1}^{(2)}}$ is an arrow that starts 
before the first arrow of $\sigma_{p}$, then $a_{1}^{(3)}a_{2}=0$ since it contains $\sigma_{p}$. In conclusion, if $a_{1}a_{2}$ 
contains a monomial relation, then:
\[\eta_{2}(1\otimes_{E}a_{1}\otimes_{E}a_{2}\otimes_{E}1)=(\sigma_{p})^{l}\otimes_{E}\sigma_{p}\otimes_{E}(\sigma_{p})^{r}.\]

\item If $a_{1}a_{2}=W_{\rho_{i}}$ for some relation $\rho_{i}\in R_{nomon}$, then
$\eta_{2}(1\otimes_{E}a_{1}\otimes_{E}a_{2}\otimes_{E}1)$ equals
\scalebox{0.9} {$s_{1}\left(\sum a_{1}a_{2}^{(1)}\otimes_{E}\overleftarrow{a_{2}^{(2)}}\otimes_{E}a_{2}^{(3)}-\sum (f_{\rho_{i}})^{(1)}\otimes_{E}\overleftarrow{(f_{\rho_{i}})^{(2)}}\otimes_{E}(f_{\rho_{i}})^{(3)}+\sum a_{1}^{(1)}\otimes_{E}\overleftarrow{a_{1}^{(2)}}\otimes_{E}a_{1}^{(3)}a_{2}\right).$}\\
The map $s_{1}$ applied to the first and second summands is zero. Applying $s_{1}$ to the third summand we conclude that
$$\eta_{2}(1\otimes_{E}a_{1}\otimes_{E}a_{2}\otimes_{E}1)=1\otimes_{E}\rho_{i}\otimes_{E}1.$$

\end{enumerate}
We summarize the information obtained as follows:

\begin{equation*}
      \eta_{2}(\lambda\otimes_{E}a_{1}\otimes_{E}a_{2}\otimes_{E}\mu)=\left\lbrace
  \begin{array}{cc}
     0&\text{ if $a_{1}a_{2}\in \mathcal{B}$,}  \\
    \lambda(\sigma_{p})^{l}\otimes_{E}\sigma_{p}\otimes_{E}(\sigma_{p})^{r}\mu&\text{if $a_{1}a_{2}\in I$,} \\
    \lambda\otimes_{E}\rho_{i}\otimes_{E}\mu&\text{if $a_{1}a_{2}=W_{\rho_{i}}$.}

\end{array}
  \right.
\end{equation*}
\vspace{1em}

We will need the explicit formula of $\eta_{3}$, for the purpose of computing it we are going to consider four different cases.

Given $a_{1}, a_{2}, a_{3}\in \overline{\mathcal{B}}$, let us consider 
$\lambda\otimes_{E}a_{1}\otimes_{E}a_{2}\otimes_{E}a_{3}\otimes_{E}\mu\in A\otimes_{E}\overline{A}^{\otimes_{E} 3}\otimes_{E}A$; 
the expression
$\eta_{3}(\lambda\otimes_{E}a_{1}\otimes_{E}a_{2}\otimes_{E}a_{3}\otimes_{E}\mu)$ is equal to: 
\[\lambda(s_{2}\eta_{2}(a_{1}\otimes_{E}a_{2}\otimes_{E}a_{3}\otimes_{E}1-
1\otimes_{E}a_{1}a_{2}\otimes_{E} a_{3}\otimes_{E}1+1\otimes_{E}a_{1}\otimes_{E}a_{2}a_{3}\otimes_{E}1-1\otimes_{E}a_{1}\otimes_{E}a_{2}\otimes_{E}a_{3}))\mu.\]
\bigbreak
\begin{enumerate}
\item We will start with the case where $a_{1}a_{2}\in \mathcal{B}$ and $a_{2}a_{3}\in \mathcal{B}$.
    Applying the definition of $\eta_{2}$ in the first and last summands, they vanish, so we obtain that 
    \[\eta_{3}(\lambda\otimes_{E}a_{1}\otimes_{E}a_{2}\otimes_{E}a_{3}\otimes_{E}\mu)=\lambda(s_{2}(-\eta_{2}(1\otimes_{E}a_{1}a_{2}\otimes_{E} a_{3}\otimes_{E}1)+\eta_{2}(1\otimes_{E}a_{1}\otimes_{E}a_{2}a_{3}\otimes_{E}1)))\mu.\]
    Observe that in this case 
    $\eta_{2}(1\otimes_{E}a_{1}a_{2}\otimes_{E} a_{3}\otimes_{E}1)=\eta_{2}(1\otimes_{E}a_{1}\otimes_{E}a_{2}a_{3}\otimes_{E}1)$ 
    since, by definition of $\eta_{2}$ both terms depend on $a_{1}a_{2}a_{3}$.

In conclusion, if $a_{1}a_{2}\in \mathcal{B}$ and $a_{2}a_{3}\in \mathcal{B}$, then:
\[\eta_{3}(\lambda\otimes_{E}a_{1}\otimes_{E}a_{2}\otimes_{E} a_{3}\otimes_{E}\mu)=0.\]

\item Let us now consider the case where $a_{1}a_{2}\in \mathcal{B}$ and $a_{2}a_{3}\in I$. In this case,
 \[\eta_{3}(\lambda\otimes_{E}a_{1}\otimes_{E}a_{2}\otimes_{E}a_{3}\otimes_{E}\mu)=\lambda(s_{2}\eta_{2}(a_{1}\otimes_{E}a_{2}\otimes_{E}a_{3}\otimes_{E}1-
1\otimes_{E}a_{1}a_{2}\otimes_{E} a_{3}\otimes_{E}1))\mu.\]
By definition of $\eta_{2}$ we obtain that
    \[\eta_{3}(\lambda\otimes_{E}a_{1}\otimes_{E}a_{2}\otimes_{E}a_{3}\otimes_{E}\mu)=\lambda(s_{2}(a_{1}(\sigma_{p})^{l}\otimes_{E}\sigma_{p}\otimes_{E}
(\sigma_{p})^{r}-
    a_{1}(\sigma_{p})^{l}\otimes_{E}\sigma_{p}\otimes_{E}(\sigma_{p})^{r}))\mu\]
    since the last relation of $a_{2}a_{3}$ starting from the left coincides with the last one of $a_{1}a_{2}a_{3}$.
    \bigbreak
    In conclusion, if $a_{1}a_{2}\in \mathcal{B}$ and $a_{2}a_{3}\in I$, then:
    \[\eta_{3}(\lambda\otimes_{E}a_{1}\otimes_{E}a_{2}\otimes_{E}a_{3}\otimes_{E}\mu)=0.\]

\item Let us continue with the symmetric case, this means $a_{1}a_{2}\in I$ and $a_{2}a_{3}\in \mathcal{B}$.
\[\eta_{3}(\lambda\otimes_{E}a_{1}\otimes_{E}a_{2}\otimes_{E}a_{3}\otimes_{E}\mu)=\lambda(s_{2}\eta_{2}(1\otimes_{E}a_{1}\otimes_{E}a_{2}a_{3}\otimes_{E}1-
1\otimes_{E}a_{1}\otimes_{E}a_{2}\otimes_{E}a_{3}))\mu.\]
Using the definition of $\eta_{2}$ in this case we obtain
\[\eta_{3}(\lambda\otimes_{E}a_{1}\otimes_{E}a_{2}\otimes_{E}a_{3}\otimes_{E}\mu)=\lambda(s_{2}((\delta_{r})^{l}\otimes_{E}\delta_{r}\otimes_{E}(\delta_{r})^{r}-
(\sigma_{q})^{l}\otimes_{E}\sigma_{q}\otimes_{E}(\sigma_{q})^{r}a_{3}))\mu\]
where $\delta_{r}$ is the last relation in $a_{1}a_{2}a_{3}$ and $\sigma_{q}$ is the last relation of $a_{1}a_{2}$ both starting from 
the left.\\
Observe that the first term vanishes since $\delta_{r}(\delta_{r})^{r}$ does not contain $2$-ambiguities.
Applying $s_{2}$ to the second term we get
        \begin{equation} \label{ecuacion}
\lambda(\sum (\sigma_{q})^{l}(\sigma_{q}\sigma_{q}^{r}a_{3})^{(1)}\otimes_{E}(\sigma_{q}\sigma_{q}^{r}a_{3})^{(2)}\otimes_{E}(\sigma_{q}\sigma_{q}^{r}a_{3})^{(3)})\mu
\end{equation}
 with $(\sigma_{q}\sigma_{q}^{r}a_{3})^{(2)}\in \mathcal{A}_{2}$. Notice that the sum is zero if there is no $2$-ambiguity contained 
 in $\sigma_{q}\sigma_{q}^{r}a_{3}.$
 
 Consider the set $\{\sigma_{1},\dots ,\sigma_{p}\}$ of the monomial relations contained in $a_{1}a_{2}a_{3}$ ordered by their starting point in $a_{1}$. For every $1\leq i\leq p-1$, we will call $w_{i}$ the $2$-ambiguity obtained from $\sigma_{i}$ and $\sigma_{i+1}$.
Rewriting \eqref{ecuacion}, if $a_{1}a_{2}\in I$ and $a_{2}a_{3}\in \mathcal{B}$ then
$$\eta_{3}(\lambda\otimes_{E}a_{1}\otimes_{E}a_{2}\otimes_{E}a_{3}\otimes_{E}\mu)=\sum_{i=q}^{p-1}\lambda(w_{i})^{l}\otimes_{E}w_{i}\otimes_{E}(w_{i})^{r}\mu.$$

\item Lastly we will analyse the case where $a_{1}a_{2}\in I$ and $a_{2}a_{3}\in I$.
Here, after applying the definition of $\eta_{2}$,
\[\eta_{3}(\lambda\otimes_{E}a_{1}\otimes_{E}a_{2}\otimes_{E}a_{3}\otimes_{E}\mu)=\lambda(s_{2}(a_{1}\delta_{r}^{l}\otimes_{E}\delta_{r}\otimes_{E}\delta_{r}^{r}-
\sigma_{q}^{l}\otimes_{E}\sigma_{q}\otimes_{E}\sigma_{q}^{r}a_{3}))\mu\]
where $\delta_{r}$ is the last relation of $a_{2}a_{3}$ and $\sigma_{q}$ is the last relation of $a_{1}a_{2}$.

If $\delta_{r}$ and $\sigma_{q}$ are disjoint, the expressions $a_{1}\delta_{r}^{l}\otimes_{E}\delta_{r}\otimes_{E}\delta_{r}^{r}$ and 
$\sigma_{q}^{l}\otimes_{E}\sigma_{q}\otimes_{E}\sigma_{q}^{r}a_{3}$ must be zero.
Now consider the case where $\delta_{r}$ and $\sigma_{q}$ are not disjoint. The first term vanishes since $\delta_{r}\delta_{r}^{r}$ 
does not contain 2-ambiguities.
Let $\{\sigma_{1},\dots ,\sigma_{p}\}$ be the set of monomial relations in $a_{1}a_{2}a_{3}$, ordered by their starting point in 
$a_{1}a_{2}$. We have,
\[\eta_{3}(\lambda\otimes_{E}a_{1}\otimes_{E}a_{2}\otimes_{E}a_{3}\otimes_{E}\mu)=\sum_{i=q}^{p-1}\lambda(w_{i})^{l}\otimes_{E}w_{i}\otimes_{E}(w_{i})^{r}\mu\]
where $w_{i}$ is the 2-ambiguity that contains the relations $\sigma_{i}$ and $\sigma_{i+1}$.\\

\end{enumerate}

Summarising,

$
      \eta_{3}(\lambda\otimes_{E}a_{1}\otimes_{E}a_{2}\otimes_{E}a_{3}\otimes_{E}\mu)=\left\lbrace
  \begin{array}{cc}
     \lambda\sum_{i=q}^{p-1}(w_{i})^{l}\otimes_{E}w_{i}\otimes_{E}(w_{i})^{r}\mu&\hbox{ if } a_{1}a_{2}\in I
     \hbox{ and }
\delta_{r}\cap \sigma_{q}\neq \emptyset,  \\
    0&\hbox{otherwise,}
  \end{array}
\right.
$
\bigbreak
where $\{w_{i}\}_{i=q,\dots, p-1}$ is the set of $2$-ambiguities contained in $a_{1}a_{2}a_{3}$, such that 
$s(w_{i})\geq s(\sigma_{q})$ with $\sigma_{q}$ the last relation of $a_{1}a_{2}$ and $\delta_{r}$ the last relation of 
$a_{1}a_{2}a_{3}$.\bigbreak

The explicit formulas of the comparison morphisms in higher degrees are usually very hard to find. In spite of that, the next 
propositions will help us later to describe $\HH^{i}(A)$ as Lie representation of $\HH^{1}(A)$ for $i\geq 2$.

\begin{proposition}\label{lema1}
Let $i\in \mathbb{N}$ and $w=u_{0}\dots u_{i-1}=w_{i-1}\dots w_{0}$ an $(i-1)$-ambiguity. The comparison morphism $\varphi_{i}$ is 
such that
\[\varphi_{i}(1\otimes_{E} w\otimes_{E} 1)=\sum_{a^{(1)}\dots a^{(i+1)}=w}1\otimes_{E}\overline{a^{(1)}}\otimes_{E}\dots \otimes_{E}\overline{a^{(i)}}\otimes_{E}a^{(i+1)}\]
where $a^{(j)}$ with $1\leq j\leq i+1$ are non zero paths in A.
\end{proposition}

\begin{proof}
For the cases $i=1, 2, 3$ we use the explicit formulas. Let us now suppose that the result holds for some even $i-1\in \mathbb{N}$. Using the formula of $\varphi_{i}$,
\begin{align*}
\varphi_{i}(1\otimes_{E}w\otimes_{E}1)&= t_{i-1}\varphi_{i-1}d_{i-1}(1\otimes_{E}w\otimes_{E}1)\\&=t_{i-1}\varphi_{i-1}(w_{i-1}\otimes_{E}w_{i-2}\dots w_{0}\otimes_{E}1)\\&-t_{i-1}\varphi_{i-1}(1\otimes_{E}u_{0}\dots u_{i-2}\otimes_{E}u_{i-1}).
\end{align*}
Observe that the second summand is zero in the $E$-reduced Bar resolution due to the inductive hypothesis and the definition of 
$t_{i-1} $.
Finally, the inductive hypothesis for $\varphi_{i-1}(1\otimes_{E}w_{i-2}\dots w_{0}\otimes_{E} 1)$ and the definition of $t_{i-1}$, give the result.

When $i-1$ is odd,
\begin{align*}
\varphi_{i}(1\otimes_{E}w\otimes_{E}1)&=t_{i-1}\varphi_{i-1}d_{i-1}(1\otimes_{E}w\otimes_{E}1)\\&
=\sum_{w}t_{i-1}(w^{(1)}\varphi_{i-1}(1\otimes_{E}w^{(2)}\otimes_{E}1)w^{(3)})\\&=\sum_{w}\sum_{b^{(1)}\dots b^{(i)}=w^{(2)}}t_{i-1}(w^{(1)}\otimes_{E}\overline{b^{(1)}}\otimes_{E}\dots \otimes_{E}\overline{b^{(i-1)}}\otimes_{E}b^{(i)}w^{(3)})\\&=\sum_{w}\sum_{b^{(1)}\dots b^{(i)}=w^{(2)}}1\otimes_{E}\overline{w^{(1)}}\otimes_{E}\overline{b^{(1)}}\otimes_{E}\dots \otimes_{E}\overline{b^{(i-1)}}\otimes_{E}b^{(i)}w^{(3)},
\end{align*}
observing that $w^{(1)}b^{(1)}\dots b^{(i)}w^{(3)}=w$, the proof is complete.
\end{proof}

\begin{proposition}\label{prop4.6}
\begin{enumerate}
\item For any $i\geq 3$ and $a_{1},\dots, a_{i}$ in $\overline{\mathcal{B}}$  
\[\eta_{i}(1\otimes_{E} a_{1}\otimes_{E} \dots \otimes_{E}a_{i}\otimes_{E} 1)=\sum_{a}\alpha_{123}a^{(1)}\otimes_{E}a^{(2)}\otimes_{E}a^{(3)}\]
with $\alpha_{123}\in \mathds{k}$, $a^{(1)}a^{(2)}a^{(3)}=a=a_{1}\dots a_{i}$ and $a^{(2)}\in \mathcal{A}_{i-1}$.

In particular, if the path $a_{1}\dots a_{i}$ does not contain $(i-1)$-ambiguities, 
\[\eta_{i}(1\otimes_{E} a_{1}\otimes_{E} \dots \otimes_{E}a_{i}\otimes_{E} 1)=0.\]
\item For all $i\geq 1$ and $w=w_{0}\dots w_{i-1}=z_{i-1}\dots z_{0}$ an $(i-1)$-ambiguity: 
\[\eta_{i}(1\otimes_{E} \overline{w_{0}}\otimes_{E} \dots \otimes_{E} \overline{w_{i-1}}\otimes_{E} 1)=1\otimes_{E} w\otimes_{E} 1.\]
\end{enumerate}
\end{proposition}
\begin{proof}
$\emph{(1)}$ The case $i=3$ has already been checked. For the inductive step,
\begin{align*}
\eta_{i}(1\otimes_{E} a_{1}\otimes_{E} \dots \otimes_{E}a_{i}\otimes_{E} 1)&=s_{i-1}\eta_{i-1}\delta_{i-1}(1\otimes_{E} a_{1}\otimes_{E} a_{2}\otimes_E\dots \otimes_{E}a_{i}\otimes_{E} 1)\\&=a_{1}s_{i-1}\eta_{i-1}(1\otimes_{E}a_{2}\otimes_E \dots \otimes_{E}a_{i}\otimes_{E} 1)\\&+\sum_{j=1}^{i-1}(-1)^{j}s_{i-1}\eta_{i-1}(1\otimes_{E}a_1\otimes_E\dots \otimes_{E}a_{j}a_{j+1}\otimes_{E}\dots\otimes_{E}1)\\&+(-1)^{i}s_{i-1}(\eta_{i-1}(1\otimes_{E}a_{1}\otimes_{E}\dots \otimes_E a_{i-1}\otimes_{E}1)a_{i}).
\end{align*}
Using the inductive hypothesis and the definition of $s_{i-1}$ in every summand, the result is obtained.
\vspace{2mm}

$\emph{(2)}$ Firstly observe that since $w_{j}, z_{k}\in \overline{\mathcal{B}}$ for $0\leq j,k\leq i-1$ it will not be necessary to 
take classes of $w_{j}$ or $z_{k}$ in $\overline{A}$.

The case where $i=1$ is clear using the definition of $\eta_{1}$.
For $i=2$, the explicit formula of $\eta_{2}$ gives
\[\eta_{2}(1\otimes_{E}w_{0}\otimes_{E} w_{1}\otimes_{E} 1)=(\sigma_{p})^{l}\otimes_{E} \sigma_{p}\otimes_{E}(\sigma_{p})^{r}\]
 where $\sigma_{p}$ is the last monomial relation from the left that appears in $w_{0}w_{1}$. Since, by  definition of ambiguity, 
 $w_{0}w_{1}$ cannot strictly contain a monomial relation, $\sigma_{p}=w_{0}w_{1}$ and the result is proven.

For the inductive step,
\begin{align*}
\eta_{i}(1\otimes_{E} w_{0}\otimes_{E} \dots \otimes_{E} w_{i-1}\otimes_{E} 1)&=s_{i-1}\eta_{i-1}\delta_{i-1}(1\otimes_{E} w_{0}\otimes_{E} \dots \otimes_{E} w_{i-1}\otimes_{E} 1)\\&=w_{0}s_{i-1}\eta_{i-1}(1\otimes_{E} \dots \otimes_{E} w_{i-1}\otimes_{E} 1)\\&+ \sum_{j=0}^{i-2}(-1)^{j+1}s_{i-1}\eta_{i-1}(1\otimes_{E} \dots \otimes_{E} w_{j}w_{j+1}\otimes_{E}\dots\otimes_{E} 1)\\&+ (-1)^{i}s_{i-1}\eta_{i-1}(1\otimes_{E} w_{0}\otimes_{E} \dots \otimes_{E} w_{i-1}).
\end{align*}
The summands in the second line vanish since $w_{j}w_{j+1}\in I$ for all $j=0,\dots ,i-2$.
Using the inductive hypothesis for the last summand it turns out that it is equal to
\[(-1)^{i}s_{i-1}(1\otimes_{E} w_{0}w_{1} \dots w_{i-2}\otimes_{E} w_{i-1})\]
and using the definition of $s_{i-1}$ the expression equals to $1\otimes_{E} w \otimes_{E} 1$, since $w$ is the only $(i-1)$-ambiguity 
contained in $w_{0}\dots w_{i-1}$.

The last step is to prove that the first summand vanishes. Using the first part of this proposition, the term 
$\eta_{i-1}(1\otimes_{E} w_{1}\otimes_{E} \dots \otimes_{E} w_{i-1}\otimes_{E} 1)$ will be a linear combination of elements of the 
form $a^{(1)}\otimes_{E} a^{(2)}\otimes_{E} a^{(3)}$ with $a=a^{(1)} a^{(2)}a^{(3)}$ a path from $s(w_{1})$ to $t(w_{i-1})$ 
and $a^{(2)}$ an $(i-2)$-ambiguity. Since the path $w_{1}$ does not start at $0$ and the quiver is toupie, the only path 
from $s(w_{1})$ to $t(w_{i-1})$ is $w_{1}\dots w_{i-1}$. The result of applying $s_{i-1}$ to the linear combination vanishes 
since $a^{(1)} a^{(2)}a^{(3)}=w_{1}\dots w_{i-1}$ and it cannot contain an $(i-1)$-ambiguity.
\end{proof}

\begin{proposition}\label{lema2}
For any toupie algebra $A$, $\eta$ is a left inverse of $\varphi$.
\end{proposition}
\begin{proof}
Again, we will prove the result by induction.
The cases $i=0$ and $i=1$ are almost immediate using the corresponding explicit formulas.
Consider $i\geq 2$. Observe that it is sufficient to verify the equality for the elements of the form 
$1\otimes_{E}u\otimes_{E}1 \in P_{i}$.

Let us check the result for $i=2$. Given a monomial relation $\sigma_{i}$,
\begin{align*}
\eta_{2}\varphi_{2}(1\otimes_{E} \sigma_{i}\otimes_{E} 1)&=\eta_{2}\left(\sum_{\sigma_{i}}1\otimes_{E}\overline{\sigma_{i}^{(1)}}\otimes_{E}\overleftarrow{\sigma_{i}^{(2)}}\otimes_{E}\sigma_{i}^{(3)}\right)\\&=
1\otimes_{E} \sigma_{i}\otimes_{E} 1
\end{align*}
since, by definition of $\eta_{2}$, the only summand which is not zero is when $\sigma_{i}^{(3)}=1$.

Let us now consider a non monomial relation of the form
 $$\rho_{i}=\alpha^{(k_{i})}+\sum_{j>k_{i}}b_{ij}\alpha^{(j)}.$$
In this case,
\begin{align*}
\eta_{2}\varphi_{2}(1\otimes_{E} \rho_{i}\otimes_{E} 1)&=\sum_{\alpha^{(k_{i})}}\eta_{2}(1\otimes_{E}\overline{(\alpha^{(k_{i})})^{(1)}}\otimes_{E}\overleftarrow{(\alpha^{(k_{i})})^{(2)}}\otimes_{E}
(\alpha^{(k_{i})})^{(3)})\\&+\sum_{j>k_{i}}b_{ij}\sum_{\alpha^{(j)}}\eta_{2}(1\otimes_{E}\overline{(\alpha^{(j)})^{(1)}}\otimes_{E}
\overleftarrow{(\alpha^{(j)})^{(2)}}\otimes_{E}
(\alpha^{(j)})^{(3)}).
\end{align*}
In the first sum, the only summand which is not zero is when
$(\alpha^{(k_{i})})^{(3)}=1$ and in this case the result is $1\otimes_{E} \rho_{i}\otimes_{E} 1$. 
The second sum is zero by definition of $\eta_{2}$, so we obtain the equality.

Now consider the case $i\geq 3$. If $w=w_{0}\dots w_{i-1}=u_{i-1}\dots u_{0}$ is an $(i-1)$-ambiguity, then
\begin{align*}
\eta_{i}\varphi_{i}(1\otimes_{E}w\otimes_{E}1)&=s_{i-1}\eta_{i-1}\delta_{i-1}t_{i-1}\varphi_{i-1}d_{i-1}(1\otimes_{E}w
\otimes_{E}1).
\end{align*}
The map $t_{*}$ is a contracting homotopy for the $E$-reduced Bar complex, $$\delta_{i-1}t_{i-1}=Id-t_{i-2}\delta_{i-2}$$ and 
the right hand side of the previous equality is
\[s_{i-1}\eta_{i-1}\varphi_{i-1}d_{i-1}(1\otimes_{E}w\otimes_{E}1)-s_{i-1}\eta_{i-1}t_{i-2}\delta_{i-2}\varphi_{i-1}d_{i-1}(1\otimes_{E}w\otimes_{E}1).\]
Since $\varphi$ is a morphism of complexes from $C_{min}(A)$ to $C_{BarE}(A)$ we know that 
$\delta_{i-2}\varphi_{i-1}d_{i-1}=\varphi_{i-2}d_{i-2}d_{i-1}=0$ and this implies that the second term is zero. 
Using the inductive hypothesis in the first term it is sufficient to prove that
\[s_{i-1}d_{i-1}(1\otimes_{E}w\otimes_{E}1)=1\otimes_{E}w\otimes_{E}1.\]
For $i$ odd,
\begin{align*}
s_{i-1}d_{i-1}(1\otimes_{E}w\otimes_{E}1)&=u_{i-1}s_{i-1}(1\otimes_{E}u_{i-2}\dots u_{0}\otimes_{E}1)\\&-s_{i-1}(1\otimes_{E}w_{0}\dots w_{i-2}\otimes_{E}w_{i-1})\\&=1\otimes_{E}w\otimes_{E}1.
\end{align*}
For $i$ even,
\begin{align*}
s_{i-1}d_{i-1}(1\otimes_{E}w\otimes_{E}1)&=s_{i-1}\left(\sum_{w^{(2)}\in V^{(i-2)}}w^{(1)}\otimes_{E}w^{(2)}\otimes_{E}w^{(3)}\right)\\&=1\otimes_{E}w\otimes_{E}1
\end{align*}
and the result is proven.
\end{proof}



\section{Gerstenhaber structure of $\HH^{*}(A)$}
\label{Gerstenhaber}
In this section we fix $\mathds{k}=\mathbb{C}$. We describe the Lie structure of $\HH^{1}(A)$, while in the next sections we will
provide the structure of $\HH^{i}(A)$, with $i\geq 2$, as Lie representations of $\HH^{1}(A)$.
\subsection{Cup product}
The cup product of the Hochschild cohomology of a toupie algebra is trivial, see \cite{GL}. For the convenience of the reader we sketch an alternative proof.
\begin{proposition}
Given a toupie algebra $A$, $\overline{f}\in \HH^{n}(A)$ and $\overline{g}\in \HH^{m}(A)$ with $m,n>0$, the cup product $\overline{f}\smile \overline{g}$ vanishes in $\HH^{m+n}(A)$.
\end{proposition}
\begin{proof}

We will prove that $\overline{f}\smile_{red}\overline{g} =\overline{0}$ in $\HH^{m+n}(A)$. 
Since $\eta^{*}$ is an isomorphism at the cohomology level, there exist $F\in \Hom_{E^{e}}(\mathds{k}\mathcal{A}_{n-1},A)$ and 
$G\in \Hom_{E^{e}}(\mathds{k}\mathcal{A}_{m-1},A)$ such that 
\[\overline{f}=\eta_{n}^{*}\overline{F}=\overline{\eta_{n}^{*}F},\]
\[\overline{g}=\eta_{m}^{*}\overline{G}=\overline{\eta_{m}^{*}G}.\]
We choose $f'$ and $g'$ such that $f'=\eta_{n}^{*}F$ and $g'=\eta_{m}^{*}G$

Let us suppose that $\overline{f'\smile_{red} g'}\neq 0$, that is $f'\smile_{red} g'$ is not a coboundary  and in particular there 
exists $a_{1}\otimes \dots \otimes a_{n+m}\in \overline{A}^{\otimes_{E}(m+n)}$ such that
\[f'\smile_{red} g'(a_{1}\otimes \dots \otimes a_{n+m})=f'(a_{1}\otimes \dots \otimes a_{n})g'(a_{n+1}\otimes \dots \otimes a_{n+m})\neq 0.\]
However, \[f'(a_{1}\otimes \dots \otimes a_{n})=\eta_{n}^{*}(F)(a_{1}\otimes \dots \otimes a_{n})=F(\eta_{n}(1\otimes a_{1}\otimes \dots \otimes a_{n}\otimes 1))\] 
and 
\[g'(a_{n+1}\otimes \dots \otimes a_{n+m})=\eta_{m}^{*}(G)(a_{n+1}\otimes \dots \otimes a_{n+m})=G(\eta_{m}(1\otimes a_{n+1}\otimes \dots \otimes a_{n+m}\otimes 1)),\] 
so $f'(a_{1}\otimes \dots \otimes a_{n})$ and $g'(a_{n+1}\otimes \dots \otimes a_{n+m})$ must be linear combinations of paths of 
positive length that start at $0$, so their cup product will be zero, which leads us to a contradiction. 
In conclusion, $\overline{f}\smile_{red} \overline{g}=\overline{f'\smile_{red} g'}=\overline{0}$.

\end{proof}
\subsection{Gerstenhaber bracket in higher degrees}
Next we will compute the Gerstenhaber bracket of two elements in the higher cohomology spaces of $A$. We will prove that the bracket 
of two cocycles of degree greater than one is always zero due to the particular shape of the quiver.
\begin{proposition}\label{asociadorcero}
Let $A=\mathds{k}Q/I$ be a toupie algebra, $\overline{f}\in \HH^{n}(A)$ and $\overline{g}\in \HH^{m}(A)$ with $m,n>1$.
The Gerstenhaber bracket $[\overline{f}, \overline{g}]_{red}$ vanishes in $\HH^{m+n-1}(A)$.
\end{proposition}
\begin{proof}
We want to prove that $[\overline{f},\overline{g}]_{red}=\overline{0}$ in $\HH^{m+n-1}(A)$. As before we choose $f'$ and $g'$ such 
that $f'=\eta_{n}^{*}F$ and $g'=\eta_{m}^{*}G$
with $F$ and $G$ obtained from the minimal resolution.
Let us suppose that there exists $a_{1}\otimes\dots\otimes a_{n+m-1}\in \overline{A}^{\otimes_{E}(m+n-1)}$ such that 
$[f',g']_{red}(a_{1}\otimes\dots\otimes a_{n+m-1})\neq 0$.
Recall that
$[f',g']_{red}=f'\circ g'-(-1)^{(m-1)(n-1)}g'\circ f'$
and that $f'\circ g'(a_{1}\otimes\dots\otimes a_{n+m-1})$ is
\[\sum_{i=1}^{n}(-1)^{(i-1)(m-1)}f'(a_{1}\otimes\cdots a_{i-1}\otimes \overline{g'(a_{i}\otimes\dots \otimes a_{i+m-1})}\otimes a_{i+m}\otimes \dots \otimes a_{n+m-1}).\]
We will verify that each summand is zero. Suppose that $\overline{g'(a_{i}\otimes\dots \otimes a_{i+m-1})}\neq 0$ in $\overline{A}$. 
Since $g'(a_{i}\otimes\dots \otimes a_{i+m-1})=G( \eta_{m}(1\otimes a_{i}\otimes\dots \otimes a_{i+m-1}\otimes 1))$ we know that it 
will be a linear combination of paths from $0$ to $\omega$ that is 
$g'(a_{i}\otimes\dots\otimes a_{i+m-1})=\sum_{\alpha^{(j)}\in {}_{0}\mathcal{B}_{\omega}}\lambda_{j}\alpha^{(j)}$. 
In this case 
$a_{1}\otimes_{E}\dots \otimes_{E}a_{i-1}\otimes_{E}\overline{\alpha^{(j)}}\otimes_{E}a_{i+m}\otimes_{E}\dots \otimes_{E}a_{n+m-1}$ 
vanishes in $\overline{A}^{\otimes_{E}n}$, so $f'\circ g'=0$. 
A similar argument proves that $g'\circ f'=0$ and we conclude that
$[\overline{f},\overline{g}]_{red}=\overline{[f',g']}_{red}=\overline{0}$ in $\HH^{m+n-1}(A)$.
\end{proof}

\subsection{Gerstenhaber bracket in $\HH^{1}(A)$}
From now on we will identify the elements of $\HH^{i}(A)$ with their classes when there is no confusion.
We will use the notation of \cite{St} that we recall.
Given $\alpha\| h\in \Hom_{E^{e}}(\mathds{k}Q_{1}, A)$ and a path $\tilde{a}$ in $A$, we will denote $\tilde{a}^{\alpha\| h}$ the sum 
of all the non zero paths obtained replacing every appearance of $\alpha$ in $\tilde{a}$ by $h$. If the path $\tilde{a}$ does not 
contain the arrow $\alpha$ or if when we replace $\alpha$ in $\tilde{a}$ by $h$ we get $0\in A$, we define 
$\tilde{a}^{\alpha\| h}= 0$. Observe that for toupie algebras any element of type $\tilde{a}^{\alpha\| h}$ will have at most 
one summand.

\begin{lemma}
Given $\alpha\| h\in \Hom_{E^{e}}(\mathds{k}Q_{1}, A)$, $b\in \overline{\mathcal{B}}$ and $f\in \Hom_{E^{e}}(\overline{A}, A)$, the 
following equalities hold:
\begin{enumerate}
\item $\eta_{1}^{\ast}(\alpha\| h)(b)=b^{\alpha\| h}$.
\item $\varphi_{1}^{\ast}(f)=\sum_{\alpha\in Q_{1}}\sum_{\alpha\|c \in Q_{1}\|\mathcal{B}}\lambda_{\alpha, c}\alpha\|c$, if 
$f(\alpha)=\sum_{c\in \mathcal{B}}\lambda_{\alpha, c}c$.
\end{enumerate}
\end{lemma}
\begin{proof}

$\emph{(1)}$ Given $\alpha\|h \in \Hom_{E^{e}}(\mathds{k}Q_{1}, A)$, we will denote by $\widetilde{(\alpha\|h)}$ the element in 
$\Hom_{A^{e}}(A\otimes_{E}\mathds{k}Q_{1}\otimes_{E}A, A)$ obtained from $\alpha\|h$ using the canonical identification of 
$\Hom_{E^{e}}(\mathds{k}Q_{1}, A)$ with $\Hom_{A^{e}}(A\otimes_{E}\mathds{k}Q_{1}\otimes_{E}A, A)$. 
For $b\in \overline{\mathcal{B}}$,
\[\eta_{1}^{\ast}(\alpha\|h)(b)=\widetilde{(\alpha\|h)}\eta_{1}(1\otimes_{E} b\otimes_{E} 1)=\widetilde{(\alpha\|h)}\sum b^{(1)}\otimes_{E} \overleftarrow{b^{(2)}}\otimes_{E} b^{(3)}=\sum b^{(1)}(\alpha\|h)(\overleftarrow{b^{(2)}})b^{(3)}.\]
\begin{itemize}
\item If $\alpha$ is an arrow contained in $b$, that is $b=\alpha^{l}\alpha \alpha^{r}$, then $\eta_{1}^{\ast}(\alpha\|h)(b)=\alpha^{l}h \alpha^{r}$.
\item If not, then $\eta_{1}^{\ast}(\alpha\|h)(b)=0$. 

The first equality is proven.

\end{itemize}

$\emph{(2)}$ Given $f\in \Hom_{E^{e}}(\overline{A}, A)$ and $\alpha\in Q_{1}$, we will denote by $\tilde{f}$ the element in 
$\Hom_{A^{e}}(A\otimes_{E}\overline{A}\otimes_{E}A, A)$ associated to $f$ using the canonical identification of 
$\Hom_{E^{e}}(\overline{A}, A)$ with $\Hom_{A^{e}}(A\otimes_{E}\overline{A}\otimes_{E}A, A)$. Now, 
\[\varphi_{1}^{*}(f)(\alpha)=\tilde{f}(\varphi_{1}(1\otimes_{E}\alpha\otimes_{E}1))
=\tilde{f}(1\otimes_{E}\alpha\otimes_{E}1)=f(\alpha).\]
If $f(\alpha)=\sum_{c\in \mathcal{B}}\lambda_{\alpha, c}c$, then
we conclude that $\varphi_{1}^{\ast}(f)=\sum_{\alpha\in Q_{1}}\sum_{\alpha\|c \in Q_{1}\|\mathcal{B}}\lambda_{\alpha, c}\alpha\|c$.
\end{proof}
The next theorem gives the formula of the Gerstenhaber bracket restricted to $\HH^{1}(A)$. The corresponding result for monomial 
algebras has been proven in \cite{St}. Even if the elements of $\HH^{1}(A)$ here differ from those studied by Strametz, the formula 
for the bracket can be written in a similar way.
\begin{theorem}\label{GerH1}
Let $A$ be a toupie algebra.
The Gerstenhaber bracket in $\HH^{1}(A)$ can be expressed in terms of the minimal resolution as:
\[[\alpha\|h, \beta\| b]=\beta\| b^{\alpha\|h}-\alpha\|h^{\beta\| b},\]
with $\alpha\|h, \beta\|b\in \Hom_{E^{e}}(\mathds{k}Q_{1}, A)$.
\end{theorem}
\begin{proof}
Given $\alpha\|h, \beta\|b\in \Hom_{E^{e}}(\mathds{k}Q_{1}, A)$, let us compute 
$[\alpha\|h, \beta\| b]\in \Hom_{E^{e}}(\mathds{k}Q_{1}, A)$ using the comparison morphisms.
Given $\gamma\in Q_{1}$,
\begin{align*}
[\alpha\|h, \beta\| b](\gamma)&=\varphi^{\ast}_{1}[\eta_{1}^{\ast}(\alpha\|h), \eta_{1}^{\ast}(\beta\|b)](\gamma)\\
&=[\eta_{1}^{\ast}(\alpha\|h), \eta_{1}^{\ast}(\beta\|b)]\varphi_{1}(1\otimes_{E}\gamma\otimes_{E}1)\\&=[\eta_{1}^{\ast}(\alpha\|h), \eta_{1}^{\ast}(\beta\|b)](\gamma).
\end{align*}

Applying now the definition of the Gerstenhaber bracket, we get:
\begin{align*}
[\eta_{1}^{\ast}(\alpha\|h), \eta_{1}^{\ast}(\beta\|b)](\gamma)&=\eta_{1}^{\ast}(\alpha\|h)\circ \eta_{1}^{\ast}(\beta\|b)(\gamma)-\eta_{1}^{\ast}(\beta\|b)\circ \eta_{1}^{\ast}(\alpha\|h)(\gamma)\\&=(\gamma^{\beta\|b})^{\alpha\|h}-(\gamma^{\alpha\|h})^{\beta\|b}.
\end{align*}
There are three cases to consider.
\begin{itemize}
\item If $\gamma=\beta$, then $\gamma^{\beta\|b}=b$, and if also $b$ contains $\alpha$, then we replace $\alpha$ by $h$ in $b$.
\item If $\gamma=\alpha$, then $\gamma^{\alpha\|h}=h$, and if also $h$ contains $\beta$, then we replace $\beta$ by $b$ in $h$.
\item If $\gamma\neq \alpha$ and $\gamma\neq\beta$, then $[\alpha\|h, \beta\| b](\gamma)=0.$
\end{itemize}
\vspace{1em}
In conclusion, $[\alpha\|h, \beta\| b]=\beta\| b^{\alpha\|h}-\alpha\|h^{\beta\| b}$.

The computation we have just made in terms of the complex induces the formula of the Gerstenhaber bracket in $\HH^{1}(A)$.
\end{proof}
\begin{remark}\label{dacero}
If $\alpha\|\alpha$ and $\beta\|\beta$ belong to $\HH^{1}(A)$, then $[\alpha\|\alpha, \beta\|\beta]=0$.
\end{remark}

\section{Decomposition of $\HH^{1}(A)$ as a Lie algebra}
\label{Lie}

In this section we will give a description of $\HH^{1}(A)$ as a Lie algebra. We will first find necessary and sufficient 
conditions for $A$ to be abelian and next we will describe in detail the centre of $\HH^{1}(A)$.
 
We will use the following notation for the explicit basis of $\HH^{1}(A)$ computed in Subsection \ref{caluloH1}:

\begin{itemize}
\item $y_{i}=\alpha^{i}_{0}\| \alpha^{i}_{0} ~~ with  ~~\alpha^{(i)}~~\text{a branch containing monomial relations}$.
\item $w_{pq}=\alpha^{(p)}\|\alpha^{(q)}$ for $p\neq q$ and $p,q=1\dots ,a$.
\item $z_{us}=\alpha^{(u)}\|\alpha^{(s)}$ for $u=1,\dots ,a$ and $\alpha^{(s)}\in  {}_{0}\mathcal{B}_{\omega}-Z$.
\item $x_{j}=\alpha^{(j)}\|\alpha^{(j)}-\alpha^{(1)}\|\alpha^{(1)}$ for $j=2\dots ,a$.
\item $t_{k}=\sum_{\alpha^{(i)} \in Q_{\rho}^{k} } \alpha^{i}_{0}\| \alpha^{i}_{0}$ for $k=1,\dots ,r$.
\end{itemize}

Recall that $C''_{1}=\{y_{i}:\alpha^{(i)}\text{ is a branch with monomial relations} \}$.\\

Using Theorem \ref{GerH1} we compute the Gerstenhaber brackets of the elements of the basis of $\HH^{1}(A)$ and obtain the 
following table:

\begin{table}[htbp]
  \centering
  \begin{tabular}{@{} c|ccccc @{}}
   & $y_{i'}$ &$x_{j'}$ & $w_{p'q'}$ & $z_{u's'}$ & $t_{k'}$   \\ 
    \hline
    $y_{i}$ & 0 & 0 & 0 & 0 & 0 \\ 
    $x_{j}$ &  & 0 & $A_{j}^{p'q'}$ & $B_{j}^{u's'}$ & 0 \\ 
    $w_{pq}$ &  &  & $E_{pq}^{p'q'}$ & $D_{pq}^{u's'}$ & 0 \\ 
    $z_{us}$ &  & &  & 0 & $-C_{k'}^{us}$ \\ 
    $t_{k}$ &  &  &  &  & 0 \\ 
 
    \hline
  \end{tabular}
  \caption{Lie bracket table}
  \label{tab:label}
\end{table}

where
\begin{itemize}
 \item $A_{j}^{p'q'}=[x_{j}, w_{p'q'}]=\delta_{j,q'}w_{p'q'}-\delta_{j,p'}w_{p'q'}-\delta_{q',1}w_{p'1}+\delta_{1,p'}w_{1q'},$
\item $B_{j}^{u's'}=[x_{j},z_{u's'}]=-\delta_{j,u'}z_{u's'}+\delta_{u',1}z_{1s'},$
\item $C_{k}^{u's'}=[t_{k},z_{u's'}]= \left\lbrace
  \begin{array}{cc}

z_{u's'} & \mbox{ if } \alpha^{(s')}\in Q_{\rho}^{k}\\ 0, & \mbox{ otherwise,}
\end{array}\right. $

\item $D_{pq}^{u's'}=[w_{pq},z_{u's'}]=-\delta_{q,u'}z_{ps'}$
and

\item $E_{pq}^{p'q'}=[w_{pq}, w_{p'q'}]=\left\lbrace
  \begin{array}{cc}
     \delta_{p,q'}w_{p'q}-\delta_{q,p'}w_{pq'}&\hbox{ if } q\neq p'
     \hbox{ or } p\neq q' \\
   \delta_{p,q'}\alpha^{(p')}\|\alpha^{(p')}-\delta_{q,p'}\alpha^{(q')}\|\alpha^{(q')} &\hbox{ if } q=p' \hbox{ and } p=q'.
  \end{array}
\right.$

\noindent In the last case, where $q=p'$ and $p=q'$, we can distinguish three different cases:

$E_{pp'}^{p'p}
=\left\lbrace
\begin{array}{ccc}
x_{p'}& \hbox{ if } p=1\\
-x_{p} & \hbox{ if } p'=1\\
x_{p}-x_{p'} & \hbox{ otherwise. }
\end{array}\right. $

\end{itemize}

\medskip

\begin{proposition}\label{a=0od>1}
Let $A$ be a toupie algebra. The Lie algebra $\HH^{1}(A)$ is abelian if and only if $a=0$ or $D\leq 1$. Moreover, if $a=0$, then the 
enveloping algebra $\mathcal{U}(\HH^{1}(A))$, is isomorphic to the polynomial algebra $\mathds{k}[x_{1},\cdots ,x_{r+m-1}]$. 
In case $a=D=1$, we have that $\mathcal{U}(\HH^{1}(A))$ is isomorphic to $\mathds{k}[y_{1},\cdots ,y_{m}]$.
\end{proposition}
\begin{proof}
For the first statement, recall from Section \ref{calculocohomo} that:
\[\Hom_{E^{e}}(\mathds{k}Q_{1},A)=\mathds{k}(Q_{1}-Z\| Q_{1}-Z)+ \mathds{k}(Z\| {}_{0}\mathcal{B}_{\omega}).\]
If $a=0$, then $Z$ is empty and $\Hom_{E^{e}}(\mathds{k}Q_{1},A)$ is generated by the elements of the form $\alpha\|\alpha$ for 
some $\alpha\in Q_{1}$. Using Remark \ref{dacero}, we conclude that $\HH^{1}(A)$ is abelian.\\
If $D\leq 1$, we distinguish two different cases:
\begin{itemize}
\item if $a=0$, we argue as before,
\item if $a=1$, then $Z=\{\alpha^{(1)}\}$ and $\Hom_{E^{e}}(\mathds{k}Q_{1},A)$ is generated by the elements of the form 
$\alpha\|\alpha$ with $\alpha\in Q_{1}-Z$ together with $\{\alpha^{(1)}\|\alpha^{(1)}\}$. Using again Remark \ref{dacero} we conclude 
that $\HH^{1}(A)$ is abelian.
\end{itemize}
For the converse, suppose that $a>0$ and $D>1$.
Again, we consider two different cases:
\begin{enumerate}
\item If $D=a$, then there are at least two elements in $Z$ that we will call $\alpha^{(1)}$ and $\alpha^{(2)}$ and the Lie bracket table gives:
\[[w_{21}, w_{12}]=-x_{2}\neq 0.\]
\item If $D> a>0$, then there exists at least one element in $Z$ that we will call $\alpha^{(1)}$ and one element in 
$_{0}\mathcal{B}_{\omega}-Z$ that we will call $\alpha^{(p)}$ which belongs to the connected component $Q_{\rho}^{k}$ for some $k$. 
Consider the elements 
$z_{1p}$ and $t_{k}$. Using again the computations in the Lie bracket table:
\[[z_{1p}, t_{k}]=-z_{1p}\neq 0.\]
\end{enumerate}
If $a=0$, Remark \ref{dimH1} implies that $dim_{\mathds{k}}\HH^{1}(A)=r+m-1$ and since $\HH^{1}(A)$ is abelian, 
the enveloping algebra $\mathcal{U}(\HH^{1}(A))$ is isomorphic to $\mathds{k}[x_{1},\cdots ,x_{r+m-1}]$. If $a=D=1$, then $r=0$ and 
$Da=1$ so again by Remark \ref{dimH1},  $dim_{\mathds{k}}\HH^{1}(A)=m$ and 
$\mathcal{U}(\HH^{1}(A))\simeq \mathds{k}[y_{1},\cdots ,y_{m}]$.
\end{proof}

From now on we will suppose that $a>0$ and $D>1$.

We will next describe the centre of the Lie algebra $\HH^{1}(A)$.
\begin{proposition}\label{centro}
Let $A$ be a toupie algebra
with $a>0$. 
The centre of the Lie algebra $\HH^{1}(A)$ is the $\mathds{k}$-vector space spanned by $C''_{1}=\{y_{i}:i=1,\dots ,m\}$.

\end{proposition}
\begin{proof}
Using the Lie bracket table we check that $C''_{1}\subset \mathcal{Z}(\HH^{1}(A))$.
Next  we will prove that every element of the centre of the algebra can be written as a linear combination of elements 
of $C''_{1}$.

Let $\Gamma$ be an element in $\mathcal{Z}(\HH^{1}(A))$. Write $\Gamma$ in terms of the given basis of $\HH^{1}(A)$ as follows:
\[\Gamma=\sum_{j}A_{j}x_{j}+\sum_{p,q}B_{p,q}w_{pq}+\sum_{u,s}C_{u,s}z_{us}+\sum_{k}D_{k}t_{k}+\sum_{i}E_{i}y_{i}.\]
Since the bracket of $\Gamma$ with any element of $C''_{1}$ is $0$, we determine the restrictions 
that this condition imposes on the coefficients of $\Gamma$.
For every $k$,
\[0=[\Gamma, t_{k}]=[\sum_{u,s} C_{u,s}z_{us}, t_{k}]=\sum_{u,s} C_{u,s}[z_{us}, t_{k}]=-\sum_{u,s} C_{u,s}z_{us}.\]
The $z_{us}$'s are linearly independent, so $C_{u,s}$ vanishes for every $s$ and for every $u=1,\dots,a$.\bigbreak
We also know that for every $j=2,\dots,a$:
\[0=[\Gamma, x_{j}]=[\Gamma, \alpha^{(j)}\| \alpha^{(j)}]-[\Gamma, \alpha^{(1)}\|\alpha^{(1)}]=
[\sum B_{p,q}w_{pq}, \alpha^{(j)}\| \alpha^{(j)}]-[\sum B_{p,q}w_{pq}, \alpha^{(1)}\| \alpha^{(1)}],\] 
using Remark \ref{dacero} and the fact that the coefficients $C_{r,s}$ vanish.
Now:
\begin{align*}
[\Gamma, x_{j}]&=-\sum_{p}(B_{p,j}w_{pj}-B_{p,1}w_{p1})+\sum_{q}(B_{j,q}w_{jq}-B_{1,q}w_{1q})\\&=
-2B_{1,j}w_{1j}+2B_{j,1}w_{j1}-\sum_{p\neq 1}B_{p,j}w_{pj}\\&~~~~+\sum_{q\neq 1}B_{j,q}w_{jq}+\sum_{p\neq j}B_{p,1}w_{p1}-\sum_{q\neq j}B_{1,q}w_{1q}.
\end{align*}
Since the $w_{pq}$'s are linearly independent, we conclude that $B_{p,q}=0$ for every $p\neq q$ and $p,q=1,\dots, a$. 

We will now see that $A_{j}=0$ for $j=2,\dots,a$.
Choosing $w_{pq}$ with $p\neq q$ and $1\leq p,q\leq a$, it turns out that
\[0=[\Gamma, w_{pq}]=\sum_{j} A_{j}[x_{j},w_{pq}].\]
If $p,q\neq 1$ then,
\[0=[\Gamma, w_{pq}]=(A_{q}-A_{p})w_{pq}\]
so we conclude that $A_{p}=A_{q}$. If $p=1$, then:
\begin{align*}
0&=[\Gamma, w_{1q}]=\sum_{j} A_{j}[\alpha^{(j)}\|\alpha^{(j)},w_{1q}]-\sum_{j} A_{j}[\alpha^{(1)}\|\alpha^{(1)},w_{1q}]\\&=A_{q}w_{1q}+\sum_{j} A_{j}w_{1q}=aA_{q}w_{1q},
\end{align*}
which implies $A_{q}=0$ for $q=2,\dots, a$.

Finally, given $u$ and $s$ such that $1\leq u\leq a$ and $\alpha^{(s)}\in {}_{0}\mathcal{B}_{\omega}-Z$ we have:
\[0=[\Gamma, z_{us}]=\sum_{k} D_{k}[t_{k},z_{us}]=\sum_{k} D_{k}[t_{k}, \alpha^{(u)}\|\alpha^{(s)}]=D_{l}z_{us}\]
if $\alpha^{(s)}$ belongs to $Q_{\rho}^{l}$.
As a consequence, $D_{k}=0$ for every $k$ and we conclude that 
\[\Gamma=\sum_{i} E_{i}y_{i}.\] 
\end{proof}
Consider the $a$-Kronecker quiver $Q_{a}$: it has only two vertices, one source and one sink, with $a$ arrows from the 
source to the sink. In the next proposition we will prove that if the toupie quiver $Q_{A}$ associated to the toupie algebra $A$ has 
$Q_{a}$ as a subquiver, then the Lie algebra $\HH^{1}(A)$ contains a subalgebra isomorphic to $\HH^{1}( \mathds{k}Q_{a})$.

\begin{proposition}\label{sl}
The Lie subalgebra of $\HH^{1}(A)$ generated by
\[\{x_{j}, w_{pq}~~with~~p\neq q;~~p,q=1,\dots, a~~and
~~j=2,\dots, a\}\]
is isomorphic to $\HH^{1}( \mathds{k}Q_{a})$.
\end{proposition}

\begin{proof}
Let $Q_{A}$ be the quiver of the toupie algebra $A$.
First note that $ \mathds{k}Q_{a}$ is a toupie algebra. Consequently, by Theorem \ref{baseH1}, the space $\HH^{1}( \mathds{k}Q_{a})$ 
has a basis:
\[\{X_{j}, W_{pq}~~with~~p\neq q;~~p,q=1,\dots, a~~and~~j=2,\dots, a\}.\]
There is a morphism of $\mathds{k}$-Lie algebras from $\HH^{1}( \mathds{k}Q_{a})$ to $\HH^{1}(A)$ that sends $X_{j}$ to 
$x_{j}$ and $W_{pq}$ to $w_{pq}$ for all $j,p,q$. This map is injective and induces an isomorphism between 
$\HH^{1}( \mathds{k}Q_{a})$ and its image.
\end{proof}

The proof of the next proposition is in \cite[Corollary 2.2.4]{S}. We will just give the explicit isomorphism since we are going to 
use it later.

From now on, suppose $\mathds{k}=\CC$.

\begin{proposition}\label{corosl}
The Lie subalgebra of $\HH^{1}(A)$ generated by the set $\{x_{j}, w_{pq}~~with~~p\neq q;~~p,q=1,\dots, a~~and~~j=2,\dots, a\}$ is 
isomorphic to $sl_{a}(\mathbb{C})$.
\end{proposition}
\begin{proof}
The isomorphism sends $w_{pq}$ to the elementary $a\times a$ matrix $E_{qp}$ and $x_{j}$ to the $a\times a$ diagonal matrix 
$E_{jj}-E_{11}$.
\end{proof}
Now we will give a decomposition of $\HH^{1}(A)$ as a Lie algebra.

\begin{theorem}\label{lie}
Let $Q$ be a toupie quiver, $Q_{\rho}$ the quiver described in Definition \ref{qro} and $\{Q_{\rho}^{h} :h=1,\dots, r\}$ the set of 
the connected components of $Q_{\rho}$.
  The $\mathds{k}$-vector space 
  \[L=\langle t_{h}, z_{us}: h=1,\dots, r;  ~~u=1,\dots, a~~ and ~~s \text{ such that }\alpha ^{(s)}\in {}_{0}\mathcal{B}_{\omega}-Z \rangle\] 
  is a solvable Lie ideal of $\HH^{1}(A)$. Moreover,
\begin{enumerate}

\item $S_{1}=\langle t_{h}: h=1,\dots, r\rangle$ is an abelian Lie subalgebra; $L_{2}=\langle z_{us}: u=1,\dots, a~~ and ~~s$
$\text{ such that }\alpha^{(s)}\in {}_{0}\mathcal{B}_{\omega}-Z\rangle$ is an abelian ideal, and:
\[L=S_{1}\ltimes L_{2}.\]
\item There is an isomorphism of Lie algebras $\HH^{1}(A)\cong \langle C''_{1}\rangle\oplus sl_{a}(\mathbb{C})\ltimes (S_{1}\ltimes L_{2}).$
\end{enumerate}

\end{theorem}
\begin{proof}
Using the Lie bracket table we know that $[t_{h}, t_{h'}]=0$,
$[t_{k}, z_{us} ]= z_{us}$ or $0$, and $[z_{us}, z_{u's'}]=0$, so we conclude that $L$ is a Lie subalgebra of $\HH^{1}(A)$. 
Moreover, since $[x_{j}, t_{k}]=0$ and $[w_{pq}, t_{k}]=0$, $L$ is also an ideal.

From the previous computations we deduce that $L$ is solvable since $[L,L]\subset L_{2}$ and $L_{2}$ is abelian.

We also deduce that $S_{1}$ is an abelian subalgebra and $L_{2}$ is an ideal. Besides, since $L_{2}$ is abelian, it is solvable.
Finally, since the projection $\pi':L\rightarrow L/L_{2}$ induces an isomorphism between $S_{1}$ and $L/L_{2}$ it turns out that
\[L=S_{1}\ltimes L_{2}.\]
On the other hand, consider the Lie algebra $\widetilde{\HH^{1}(A)}:=\HH^{1}(A)/\mathcal{Z}(\HH^{1}(A))$
and the projection $\pi:\widetilde{\HH^{1}(A)}\rightarrow \widetilde{\HH^{1}(A)}/L$. 
Propositions \ref{sl} and \ref{corosl} provide an isomorphism between $sl_{a}(\mathbb{C})$ and $\widetilde{\HH^{1}(A)}/L$, 
which implies that:
\[\widetilde{\HH^{1}(A)}\simeq sl_{a}(\mathbb{C})\ltimes L\]
Using this isomorphism and Lemma \ref{centro}, 
we obtain the decomposition of $\HH^{1}(A)$.
\end{proof}

\begin{corollary}\label{rad}
Let $A$ be a toupie algebra. The radical of the Lie algebra $\HH^{1}(A)$ decomposes as follows:
\[rad(\HH^{1}(A))=L\oplus \langle C''_{1}\rangle.\]
\end{corollary}
\begin{proof}
Since $L$ is a solvable ideal and $\langle C''_{1}\rangle $ is abelian, it follows that $L\oplus \langle C''_{1}\rangle$ is a solvable ideal. 
Besides, the quotient of $\HH^{1}(A)$ by $L\oplus \langle C''_{1}\rangle$ is isomorphic  to $sl_{a}(\mathbb{C})$, which is semisimple.
\end{proof}
The next corollary gives a similar decomposition to the one given in Theorem \ref{lie}. It can be easily proven using the 
previous corollary and the Levi decomposition.

\begin{corollary}\label{lie2}
Let $A$ be a toupie algebra. The algebra $\HH^{1}(A)$ decomposes as follows:
\[\HH^{1}(A)=(L\oplus \langle C''_{1}\rangle)\rtimes sl_{a}(\mathbb{C}).\]
\end{corollary}

\begin{corollary}\label{D=a,m=0}
Let $A$ be a toupie algebra such that $\HH^{1}(A)$ is not abelian. The Lie algebra $\HH^{1}(A)$ is semisimple if and only if $a=D$ 
and $m=0$.
\end{corollary}
\begin{proof}

If $\HH^{1}(A)$ is semisimple, then $rad(\HH^{1}(A))=0$ and $\# C''_{1}=m=0$. Using Corollary \ref{rad}, 
we know that $\HH^{1}(A)$ is isomorphic to $sl_{a}(\mathbb{C})$, via the identification of $sl_{a}(\mathbb{C})$ with 
$\langle x_{j}, w_{pq}~~:p\neq q;~~p,q=1,\dots, a~~and~~j=2,\dots, a\rangle$ as in Corollary \ref{lie2}.
Since $\HH^{1}(A)$ is not abelian, Theorem \ref{a=0od>1} implies that there exists at least one arrow from $0$ to 
$\omega$ that we will call $\alpha^{(1)}\in Z$.
If $a$ was strictly smaller than $D$, then it would exist a branch $\alpha^{(i)}\in {}_{0}\mathcal{B}_{\omega}-Z$. 
Consequently, the element $z_{1i}=\alpha^{(1)}\|\alpha^{(i)}$ would be non zero in $\HH^{1}(A)$ and that leads to a contradiction.

Let us prove the converse. We know that if $a=D$ and $m=0$, then the basis of $\HH^{1}(A)$ does not contain elements of the form 
$y_{i}$, $z_{us}$ or $t_{h}$ and it follows that $\HH^{1}(A)\simeq sl_{a}(\mathbb{C})$.
\end{proof}

\section{$\HH^{i}(A)$ as representation of $\HH^{1}(A)$}

\subsection{$\HH^{2}(A)$ as a Lie representation of $\HH^{1}(A)$}
\label{H2comorep}

In this section we will describe the action of the Lie algebra $\HH^{1}(A)$ on $\HH^{2}(A)$.
Given a path $c$, we will call $\overleftarrow{{}_{1}c}$ the first arrow of $c$.
Consider $\alpha\| b\in \HH^{1}(A)$, $\rho\| c\in \HH^{2}(A)$ and a monomial relation $\sigma$ belonging to a minimal set of generators of $\emph{I}$. Let us compute $[\alpha\| b, \rho\| c]\in \HH^{2}(A)$:
\begin{align*}
[\alpha\| b, \rho\| c](\sigma)&=\varphi_{2}^{\ast}[\eta_{1}^{\ast}(\alpha\| b), \eta_{2}^{\ast}(\rho\| c)](\sigma)\\&=\sum_{\sigma}[\eta_{1}^{\ast}(\alpha\| b), \eta_{2}^{\ast}(\rho\| c)](1\otimes \overline{\sigma^{(1)}}\otimes \overleftarrow{\sigma^{(2)}}\otimes \sigma^{(3)})
\\&=\sum_{\sigma}(\eta_{1}^{\ast}(\alpha\| b)\circ \eta_{2}^{\ast}(\rho\| c)(\overline{\sigma^{(1)}}\otimes \overleftarrow{\sigma^{(2)}}) )\sigma^{(3)}\\&~~~~-\sum_{\sigma}\eta_{2}^{\ast}(\rho\| c)(\overline{\eta_{1}^{\ast}(\alpha\| b)( \overline{\sigma^{(1)}})}\otimes \overleftarrow{\sigma^{(2)}})\sigma^{(3)}\\&~~~~-\sum_{\sigma}\eta_{2}^{\ast}(\rho\| c)(\overline{\sigma^{(1)}}\otimes \overline{\eta_{1}^{\ast}(\alpha\| b)(\overleftarrow{\sigma^{(2)}})})\sigma^{(3)},
\end{align*}
where the sums range over all possible decompositions of $\sigma$ with $ \overleftarrow{\sigma^{(2)}}$ an arrow.
Let us first compute the first summand:

\begin{align*}
\sum_{\sigma}(\eta_{1}^{\ast}(\alpha\| b)\circ \eta_{2}^{\ast}(\rho\| c)( \overline{\sigma^{(1)}}\otimes \overleftarrow{\sigma^{(2)}})) \sigma^{(3)}&=\sum_{\sigma}\eta_{1}^{\ast}(\alpha\| b)(\overline{ (\rho\| c)(\eta_{2}(1\otimes\overline{\sigma^{(1)}}\otimes \overleftarrow{\sigma^{(2)}}\otimes 1)))\sigma^{(3)}}.
\end{align*}

Since the monomial relation $\sigma$ does not contain any other relation, the term $\eta_{2}(1\otimes
\overline{\sigma^{(1)}}\otimes \overleftarrow{\sigma^{(2)}}\otimes 1)$ vanishes except for the case where $\sigma^{(3)}$ is $1$,
thus
\begin{align*}
 \sum_{\sigma}\eta_{1}^{\ast}(\alpha\| b)( \overline{(\rho\| c)(\eta_{2}(1\otimes \overline{\sigma^{(1)}}\otimes \overleftarrow{\sigma^{(2)}}\otimes 1)))\sigma^{(3)}}&=\eta_{1}^{\ast}(\alpha\| b) \overline{(\rho\| c)(\eta_{2}(1\otimes \overline{\sigma^{(1)}}\otimes \overleftarrow{\sigma^{(2)}}\otimes 1))}\\&=\eta_{1}^{\ast}(\alpha\| b) \overline{(\rho\| c)(\sigma)}\\&=\delta_{\rho, \sigma}\eta_{1}^{\ast}(\alpha\| b) (\overline{c})\\&=\delta_{\rho, \sigma}\alpha\| b(\eta_{1} (1\otimes \overline{c}\otimes 1))\\&=\delta_{\rho, \sigma}\sum_{c}\alpha\| b(c^{(1)}\otimes \overleftarrow{c^{(2)}}\otimes c^{(3)})
 \\&=\left\{
       \begin{array}{ll}
         \delta_{\rho, \sigma}c^{\alpha\|b}& \hbox{if $\alpha=\overleftarrow{{}_{1}c}$;} \\
         0 & \hbox{otherwise.}
       \end{array}
     \right.
 \end{align*}
As for the second summand,\\
$\displaystyle -\sum_{\sigma}\eta_{2}^{\ast}(\rho\| c)(\overline{\eta_{1}^{\ast}(\alpha\| b)(\overline{\sigma^{(1)}})}\otimes \overleftarrow{\sigma^{(2)}})\sigma^{(3)}=\left\{
                                                                                                                                   \begin{array}{ll}
                                                                                                                                    \displaystyle -\sum_{\sigma}(\rho\| c)\eta_{2}(1\otimes \overline{\sigma^{(1)}}\otimes \overleftarrow{\sigma^{(2)}}\otimes 1)\sigma^{(3)}  & \hbox{if $\alpha=\overleftarrow{{}_{1}\sigma}$;} \\
                                                                                                                                     \quad 0 & \hbox{otherwise,}
                                                                                                                                   \end{array}
                                                                                                                                 \right.
$
\break

\noindent since if $\alpha\neq \overleftarrow{{}_{1}\sigma}$ then $\eta_{1}^{\ast}(\alpha\| b)(\overline{\sigma^{(1)}})=(\alpha\| b)(\eta_{1}(1\otimes \overline{\sigma^{(1)}}\otimes 1))=0$.
For the computation in the first line observe that if $\alpha= \overleftarrow{{}_{1}\sigma}$, then $b = \alpha$.

\noindent Taking into account that $\sigma$ does not contain any other relation, we know that $\eta_{2}(1\otimes \overline{\sigma^{(1)}}\otimes \overleftarrow{\sigma^{(2)}}\otimes 1)$ vanishes except for the case where $\sigma^{(3)}=1$; in that case,
$$-\sum_{\sigma}(\rho\| c)\eta_{2}(1\otimes \overline{\sigma^{(1)}}\otimes \overleftarrow{\sigma^{(2)}}\otimes 1)\sigma^{(3)}=-(\rho\| c)\eta_{2}(1\otimes \overline{\sigma^{(1)}}\otimes\overleftarrow{\sigma^{(2)}}\otimes 1)=-(\rho\| c)(\sigma)=-\delta_{\sigma,\rho}c.$$
\bigbreak

Finally, we will prove that the last summand vanishes. For this, notice that \\
$\displaystyle -\sum_{\sigma}\eta_{2}^{\ast}(\rho\| c)(1\otimes \overline{\sigma^{(1)}}\otimes\overline{ \eta_{1}^{\ast}(\alpha\| b)(\overleftarrow{\sigma^{(2)}}})\otimes 1)\sigma^{(3)}=\left\{
   \begin{array}{ll}
     -\eta_{2}^{\ast}(\rho\| c)(1\otimes \overline{1}\otimes \overleftarrow{\sigma^{(1)}}\otimes 1)\sigma^{(2)} & \hbox{if $\alpha=\overleftarrow{{}_{1}\sigma}$;} \\
     0 & \hbox{otherwise.}
   \end{array}
 \right.
$\\
and observe that $1\otimes \overline{1}\otimes \overleftarrow{\sigma^{(1)}}\otimes 1$ is zero in the $E$-reduced Bar resolution.
\medskip

\noindent Given a non monomial relation,  which is necessarily  of the form 
\[\rho_{i}=\alpha^{(k_{i})}+\sum_{j> k_{i}} b_{ij}\alpha^{(j)} \text{ for some $i$}\]
 and using the same notations,
\begin{align*}
[\alpha\| b, \rho\| c](\rho_{i})&=\varphi_{2}^{\ast}[\eta_{1}^{\ast}(\alpha\| b), \eta_{2}^{\ast}(\rho\| c)](\rho_{i})
\\&=\sum_{\alpha^{(k_{i})}}[\eta_{1}^{\ast}(\alpha\| b), \eta_{2}^{\ast}(\rho\| c)](1\otimes\overline{ (\alpha^{(k_{i})})^{(1)}}\otimes \overleftarrow{(\alpha^{(k_{i})})^{(2)}}\otimes (\alpha^{(k_{i})})^{(3)})\\&~~~~+\sum_{j> k_{i}} b_{ij}\sum_{\alpha^{(j)}}[\eta_{1}^{\ast}(\alpha\| b), \eta_{2}^{\ast}(\rho\| c)](1\otimes \overline{(\alpha^{(j)})^{(1)}}\otimes \overleftarrow{(\alpha^{(j)})^{(2)}}\otimes (\alpha^{(j)})^{(3)})
\\&=\sum_{\alpha^{(k_{i})}}\eta_{1}^{\ast}(\alpha\| b)\circ \eta_{2}^{\ast}(\rho\| c)(1\otimes\overline{ (\alpha^{(k_{i})})^{(1)}}\otimes \overleftarrow{(\alpha^{(k_{i})})^{(2)}}\otimes (\alpha^{(k_{i})})^{(3)})\\&~~~~+\sum_{j> k_{i}}b_{ij}\sum_{\alpha^{(j)}}\eta_{1}^{\ast}(\alpha\| b)\circ \eta_{2}^{\ast}(\rho\| c)(1\otimes \overline{(\alpha^{(j)})^{(1)}}\otimes \overleftarrow{(\alpha^{(j)})^{(2)}}\otimes (\alpha^{(j)})^{(3)})\\&~~~~ -\sum_{\alpha^{(k_{i})}}\eta_{2}^{\ast}(\rho\| c)(\overline{\eta_{1}^{\ast}(\alpha\| b)((\overline{\alpha^{(k_{i})})^{(1)}})}\otimes \overleftarrow{(\alpha^{(k_{i})})^{(2)}})(\alpha^{(k_{i})})^{(3)}\\&~~~~  -\sum_{j> k_{i}}b_{ij}\sum_{\alpha^{(j)}}\eta_{2}^{\ast}(\rho\| c)(\overline{\eta_{1}^{\ast}(\alpha\| b)(\overline{(\alpha^{(j)})^{(1)}})}\otimes \overleftarrow{(\alpha^{(j)})^{(2)}})(\alpha^{(j)})^{(3)}\\&~~~~-\sum_{\alpha^{(k_{i})}}\eta_{2}^{\ast}(\rho\| c)((\overline{\alpha^{(k_{i})})^{(1)}}\otimes \overline{\eta_{1}^{\ast}(\alpha\| b)(\overleftarrow{(\alpha^{(k_{i})})^{(2)}})})(\alpha^{(k_{i})})^{(3)}\\&~~~~-\sum_{j> k_{i}}b_{ij}\sum_{\alpha^{(j)}}\eta_{2}^{\ast}(\rho\| c)((\overline{\alpha^{(j)})^{(1)}}\otimes\overline{ \eta_{1}^{\ast}(\alpha\| b)(\overleftarrow{(\alpha^{(j)})^{(2)}})})(\alpha^{(j)})^{(3)}.
\end{align*}

Let us compute the first summand:
\begin{align*}
~~~~&\sum_{\alpha^{(k_{i})}}\eta_{1}^{\ast}(\alpha\| b)\circ \eta_{2}^{\ast}(\rho\| c)(1\otimes\overline{ (\alpha^{(k_{i})})^{(1)}}\otimes \overleftarrow{(\alpha^{(k_{i})})^{(2)}}\otimes (\alpha^{(k_{i})})^{(3)})\\=&\sum_{\alpha^{(k_{i})}}\eta_{1}^{\ast}(\alpha\| b)(\overline{ (\rho\| c)(\eta_{2}(1\otimes\overline{{(\alpha^{(k_{i})})}^{(1)}}\otimes\overleftarrow{(\alpha^{(k_{i})})^{(2)}}\otimes 1))(\alpha^{(k_{i})})^{(3)}}).
\end{align*}

The term $\eta_{2}(1\otimes\overline{(\alpha^{(k_{i})})^{(1)}}\otimes \overleftarrow{(\alpha^{(k_{i})})^{(2)}}\otimes 1)$ vanishes except when $(\alpha^{(k_{i})})^{(3)}=1$, and in this case, by definition of $\eta_{2}$, the result equals $\rho_{i}$.\\
Hence,
\begin{align*}
 \sum_{\alpha^{(k_{i})}}\eta_{1}^{\ast}(\alpha\| b)(\overline{ (\rho\| c)(\eta_{2}(1\otimes\overline{(\alpha^{(k_{i})})^{(1)}}\otimes  \overleftarrow{(\alpha^{(k_{i})})^{(2)}}\otimes 1))(\alpha^{(k_{i})})^{(3)}})&=\eta_{1}^{\ast}(\alpha\| b) \overline{(\rho\| c)(\rho_{i})}\\&=\delta_{\rho, \rho_{i}}\eta_{1}^{\ast}(\alpha\| b) (\overline{c})\\&=\delta_{\rho, \rho_{i}}(\alpha\| b)(\eta_{1} (1\otimes \overline{c}\otimes 1))\\&=\left\{
                                                                                                   \begin{array}{ll}
                                                                                                     \delta_{\rho, \rho_{i}}c^{\alpha\| b} & \hbox{$\alpha=\overleftarrow{{}_{1}c}$;} \\
                                                                                                     0 & \hbox{otherwise.}
                                                                                                   \end{array}
                                                                                                 \right.
.
 \end{align*}

Let us verify that the second summand is zero:
\begin{align*}
&\sum_{j> k_{i}}b_{ij}\sum_{\alpha^{(j)}}\eta_{1}^{\ast}(\alpha\| b)\circ \eta_{2}^{\ast}(\rho\| c)(1\otimes \overline{(\alpha^{(j)})^{(1)}}\otimes \overleftarrow{(\alpha^{(j)})^{(2)}}\otimes (\alpha^{(j)})^{(3)})
\\&=\sum_{j> k_{i}}b_{ij}\sum_{\alpha^{(j)}}\eta_{1}^{\ast}(\alpha\| b)(\overline{ (\rho\| c)(\eta_{2}(1\otimes\overline{(\alpha^{(j)})^{(1)}}\otimes \overleftarrow{(\alpha^{(j)})^{(2)}}\otimes 1))(\alpha^{(j)})^{(3)}})\\&=0
\end{align*}
since $\eta_{2}(1\otimes\overline{(\alpha^{(j)})^{(1)}}\otimes \overleftarrow{(\alpha^{(j)})^{(2)}}\otimes 1)=0$ for every $j>k_{i}$.\\

To deal with the third summand, we need to distinguish two different cases.
\begin{enumerate}

\item If $\alpha$ is not the first arrow of $\alpha^{(k_{i})}$ then:
\[-\sum_{\alpha^{(k_{i})}}\eta_{2}^{\ast}(\rho\| c)(\overline{\eta_{1}^{\ast}(\alpha\| b)((\overline{\alpha^{(k_{i})})^{(1)}}})\otimes \overleftarrow{(\alpha^{(k_{i})})^{(2)}})(\alpha^{(k_{i})})^{(3)}=0.\]
\item If $\alpha$ is the first arrow of $\alpha^{(k_{i})}$, which implies that $b=\alpha$:
\begin{align*}
&-\sum_{\alpha^{(k_{i})}}\eta_{2}^{\ast}(\rho\| c)(\overline{\eta_{1}^{\ast}(\alpha\| \alpha)((\overline{\alpha^{(k_{i})})^{(1)}}})\otimes \overleftarrow{(\alpha^{(k_{i})})^{(2)}})(\alpha^{(k_{i})})^{(3)}\\&=-\sum_{\alpha^{(k_{i})}}\eta_{2}^{\ast}(\rho\| c)(\overline{(\alpha\| \alpha)(\eta_{1}(1\otimes\overline{(\alpha^{(k_{i})})^{(1)}}\otimes 1)}\otimes \overleftarrow{(\alpha^{(k_{i})})^{(2)}})(\alpha^{(k_{i})})^{(3)}\\&=-\sum_{\alpha^{(k_{i})}}\eta_{2}^{\ast}(\rho\| c)((\overline{\alpha^{(k_{i})})^{(1)}}\otimes \overleftarrow{(\alpha^{(k_{i})})^{(2)}})(\alpha^{(k_{i})})^{(3)}\\&=-\sum_{\alpha^{(k_{i})}}(\rho\| c)\eta_{2}(1\otimes(\overline{\alpha^{(k_{i})})^{(1)}}\otimes\overleftarrow{(\alpha^{(k_{i})})^{(2)}}\otimes 1)(\alpha^{(k_{i})})^{(3)}.
\end{align*}
  The term $\eta_{2}(1\otimes(\overline{\alpha^{(k_{i})})^{(1)}} \otimes \overleftarrow{(\alpha^{(k_{i})})^{(2)}}\otimes 1)$ vanishes except for $(\alpha^{(k_{i})})^{(3)}=1$, and in this 
  case $\eta_{2}(1\otimes(\overline{\alpha^{(k_{i})})^{(1)}}\otimes \overleftarrow{(\alpha^{(k_{i})})^{(2)}}\otimes 1)=\rho_{i}$ .
Therefore,
\begin{align*}
-\sum_{\alpha^{(k_{i})}}(\rho\| c)\eta_{2}(1\otimes(\overline{\alpha^{(k_{i}}))^{(1)}}\otimes \overleftarrow{(\alpha^{(k_{i})})^{(2)}}\otimes 1)(\alpha^{(k_{i})})^{(3)}&=-(\rho\| c)\eta_{2}(1\otimes(\overline{\alpha^{(k_{i})})^{(1)}}\otimes \overleftarrow{(\alpha^{(k_{i})})^{(2)}}\otimes 1)\\&=-(\rho\| c)(\rho_{i})\\&=-\delta_{\rho,\rho_{i}}c.
\end{align*}
\end{enumerate}
An analogous computation shows that the fourth summand is zero. The situation is even easier, since in this case
 $\eta_{2}(1\otimes(\overline{\alpha^{(j)})^{(1)}}\otimes \overleftarrow{(\alpha^{(j)})^{(2)}}\otimes 1)=0$ for $j>k$. 
 
\medskip

\noindent Finally we verify that the last two summands vanish. For the fifth summand observe that \\
$-\sum_{\alpha^{(k_{i})}}\eta_{2}^{\ast}(\rho\| c)((\overline{\alpha^{(k_{i})})^{(1)}}\otimes \overline{\eta_{1}^{\ast}(\alpha\| b)(\overleftarrow{(\alpha^{(k_{i})})^{(2)}})})(\alpha^{(k_{i})})^{(3)}$ equals $ -(\rho\| c)(\eta_{2}(1\otimes\overline{1}\otimes \overleftarrow{(\alpha^{(k_{i})})^{(2)}}\otimes 1 ))(\alpha^{(k_{i})})^{(3)}$  if $\alpha$ is the first arrow of $\alpha^{(k_{i})}$ and it is zero otherwise.

\noindent From the fact that $1\otimes\overline{1}\otimes \overleftarrow{(\alpha^{(k_{i})})^{(2)}}\otimes 1$ is zero in the $E$-reduced Bar resolution,
we conclude that the fifth summand also vanishes.
\bigbreak
\noindent The last summand also vanishes for analogous reasons.\bigbreak
Summarising, given $\alpha\|b\in \HH^{1}(A)$ and $\rho\|c\in \HH^{2}(A)$:

\begin{equation}\label{eq:corchete}
[\alpha\|b, \rho\|c]=\delta_{\alpha, \overleftarrow{{}_{1}c}}\rho\|c^{\overleftarrow{{}_{1}c}\|b}-\delta_{\alpha,\overleftarrow{{}_{1}W_{\rho}}}\rho\|c.
\end{equation}


\medskip 

Next we are going to describe the action of $sl_{a}(\mathbb{C})$ as a Lie subalgebra of $\HH^{1}(A)$ in $\HH^{2}(A)$.

\begin{theorem}\label{descH2sla}
Let $A$ be a toupie algebra with $a>0$, and let 
\[E=\{\rho\|\alpha^{(k)}: \text{ $\rho$ is a relation from } 0 \text{ to } \omega \text{ and }\alpha^{(k)}\in {}_{0}\mathcal{B}_{\omega}-Z \}\] 
be a subset of the basis of $\HH^{2}(A)$. 

The space $\HH^{2}(A)$ decomposes, as representation of $sl_{a}(\mathbb{C})$, in a direct sum as follows:\\
\[\HH^{2}(A)\simeq\bigoplus_{dim_{ \mathbb{C}} <E>} V_{0}~~\oplus \bigoplus_{\rho:s(\rho)=0, t(\rho)=\omega} V\] 
where $V_{0}\simeq \mathbb{C}$ is the trivial representation and $V$ is the standard representation of $sl_{a}(\mathbb{C})$.
\end{theorem}

\begin{proof}
Recall that $D_{2}=0$ and so $Ker D_{2}=\Hom_{E^{e}}(\mathds{k}\mathcal{R}, A)$. Observe that
$\HH^{2}(A)$ is generated by the classes of the elements of the set 
$ \{\rho\|\alpha^{(i)}:\text{ $\rho$ is a relation from } 0 \text{ to $\omega$ and }\alpha^{(i)}\in Z\}\cup\{\rho\|\alpha^{(j)}:\text{ $\rho$ is a relation from } 0 \text{ to $\omega$ and }\alpha^{(j)}\in {}_{0}\mathcal{B}_{\omega}-Z\}.$

Using Proposition \ref{sl} and Proposition \ref{corosl}, we identify the generators of $sl_{a}(\mathbb{C})$ with the set 
$\{x_{j}, w_{pq}~~with~~p\neq q;~~p,q=1,\dots, a~~and~~j=2,\dots, a\}$ in $\HH^{1}(A)$. Now identify the set \[\mathfrak{h}=\{b_{1}\alpha^{(1)}\|\alpha^{(1)}+b_{2}\alpha^{(2)}\|\alpha^{(2)}+\dots +b_{a}\alpha^{(a)}\|\alpha^{(a)}: b_{1}+b_{2}+\dots +b_{a}=0\}=\langle x_{j}: j=2,\dots, a\rangle\] 
with the diagonal matrices with trace zero in $sl_{a}(\mathbb{C})$. Using the invariance of the Jordan decomposition 
--see Theorem 9.20 in \cite{FH}-- we know that $\mathfrak{h}$ acts diagonally on $\HH^{2}(A)$. 
This means that $\HH^{2}(A)=\oplus V_{\lambda}$ where $\lambda$ runs over $\mathfrak{h}^{*}$ and $V_{\lambda}=\{v\in V: h.v=\lambda(h)v\hbox{ for all }h\in \mathfrak{h}\}$
 is the weight space of weight $\lambda$.
 For every $i$, $1\leq i\leq a$, let $L_{i}\in \mathfrak{h}^{*}$ be defined by $L_{i}(b_{1}\alpha^{(1)}\|\alpha^{(1)}+\dots +b_{a}\alpha^{(a)})=b_{i}$.
Notice that $\mathfrak{h}^{*}=\mathbb{C}\{L_{1}, \dots, L_{a}\}/\langle L_{1}+\dots +L_{a}\rangle $.

Given $b_{1}\alpha^{(1)}\|\alpha^{(1)}+\dots +b_{a}\alpha^{(a)}\|\alpha^{(a)}\in \mathfrak{h}$, $\rho$ a relation from $0$ to $\omega$ and $\alpha^{(i)}\in Z$, 
and $i$ such that $1\le i \le a$, the equality (\ref{eq:corchete}) implies that
\[[b_{1}\alpha^{(1)}\|\alpha^{(1)}+\dots +b_{a}\alpha^{(a)}\|\alpha^{(a)}, \rho\| \alpha^{(i)}]=b_{i}\rho\| \alpha^{(i)},\]
therefore $\rho\| \alpha^{(i)}$ belongs to the weight space associated to the weight $L_{i}$.
Let us see that $L_{1}$ is a maximal weight of $\HH^{2}(A)$. For that, again by (\ref{eq:corchete}), it is enough to observe that given 
$w_{ij}=\alpha^{(i)}\| \alpha^{(j)}$ with $j< i$, the bracket $[w_{ij}, \rho\| \alpha^{(1)}]$ is zero. 
Using existence and uniqueness theorems --see \cite{Hu} Chapter VI Theorems A and B-- and the fact that $L_{1}$ is the unique maximal weight in the standard representation, 
we conclude that the irreducible representation generated by $\rho\| \alpha^{(1)}$, that turns out to be 
$\langle \{\rho\|\alpha^{(i)}:\text{ $\rho$ is a relation from $0$ to $\omega$ and }\alpha^{(i)}\in Z\}\rangle$, 
is isomorphic to the standard representation for each $\rho$ from $0$ to $\omega$.

On the other hand, if $\alpha^{(i)}\in {}_{0}\mathcal{B}_{\omega}-Z$, then
\[[b_{1}\alpha^{(1)}\|\alpha^{(1)}+\dots +b_{a}\alpha^{(a)}\|\alpha^{(a)}, \rho\| \alpha^{(i)}]=0\]
and this implies that if $\overline{\rho\| \alpha^{(i)}}$ is not zero in $\HH^{2}(A)$, 
then it belongs to the weight space associated to the weight $0$. 
Besides, in this case, the subrepresentation generated by $\rho\| \alpha^{(i)}$, which is $\mathbb{C}.\rho\| \alpha^{(i)}$, is isomorphic to the trivial representation and we obtain the decomposition we were looking for.
\end{proof}

\begin{remark}
Given a relation $\rho$ from $0$ to $\omega$, the $\mathds{k}$-vector space $V_{\rho}=\langle\rho\|\alpha^{(k)}:\alpha^{(k)}\in{}_{0}\mathcal{B}_{\omega}\rangle$ is a Lie subrepresentation of $\HH^{1}(A)$ in $\HH^{2}(A)$.
\end{remark}

\begin{theorem}\label{descH2}
Let $A$ be a toupie algebra such that $a$ is positive. The decomposition of $\HH^{2}(A)$ as Lie representation of $\HH^{1}(A)$ is
\[\HH^{2}(A)=\bigoplus_{\rho:s(\rho)=0, t(\rho)=\omega}V_{\rho},\]
where $V_{\rho}=\langle\rho\|\alpha^{(k)}:\alpha^{(k)}\in{}_{0}\mathcal{B}_{\omega}\rangle $. Besides, for every relation $\rho$ from $0$ to $\omega$, the $\HH^{1}(A)$-module $V_{\rho}$ 
is indecomposable.
Moreover,
if $D=a$, then $V_{\rho}$ is irreducible for every $\rho$. The converse holds if there exists a monomial relation from $0$ to $\omega$.
\end{theorem}

\begin{proof}
It is clear that $\displaystyle \HH^{2}(A)=\sum_{\rho:s(\rho)=0, t(\rho)=\omega}\!\!V_{\rho}\ $ and that $\ V_{\rho}\cap \Sigma_{\rho'\neq \rho}V_{\rho'}=\{0\}$ for every $\rho$ from $0$ 
to $\omega$.
Let us now see that $V_{\rho}$ is indecomposable for every relation $\rho$ from $0$ to $\omega$.
We know that, since $a$ is positive, there exists at least one arrow from $0$ to $\omega$ which will be called $\alpha^{(1)}$. We assert that $\rho\|\alpha^{(1)}$ is not zero 
in $\HH^{2}(A)$ since $\rho\|\alpha^{(1)}$ does not belong to $Im(D_{1})$. In the particular case where $D=1$, we have $dim V_{\rho}=1$ 
and $V_{\rho}$ es indecomposable.
In case $D>1$, given $\alpha^{(i)}\in {}_{0}\mathcal{B}_{\omega}$ with $i\neq 1$:
\[[\alpha^{(1)}\|\alpha^{(i)}, \rho\|\alpha^{(1)}]=\rho\|\alpha^{(i)}.\]
This equality implies that the orbit of the action of $\HH^{1}(A)$ on $\rho\|\alpha^{(1)}$ is $V_{\rho}$ and then $V_{\rho}$ is indecomposable.

We will next prove that if $D=a$, then $V_{\rho}$ is irreducible for every relation $\rho$ from $0$ to $\omega$.
Since $D=a$, using Theorem \ref{lie}, we know that $\HH^{1}(A)=A''_{1}\bigoplus sl_{a}(\mathbb{C})$. Let us call $\widetilde{V_{\rho}}$ the representation of the Lie 
subalgebra $sl_{a}(\mathbb{C})$ with underlying vector space $V_{\rho}$. Using Theorem \ref{descH2sla}, there is an isomorphism $\widetilde{V_{\rho}}\simeq V$ where $V$ is the 
standard representation of $sl_{a}(C)$ and as we already know, it is irreducible. It remains to prove that $V_{\rho}$ is also irreducible. Any non trivial subrepresentation 
of $\HH^{2}(A)$ contained in $V_{\rho}$ would also be a non zero subrepresentation of $sl_{a}(\mathbb{C})$ contained in $\widetilde{V_{\rho}}$ and that is absurd.
Finally, we will prove that if there exists a monomial relation $\rho$ from $0$ to $\omega$ and $V_{\rho}$ is irreducible, then $D=a$.
Suppose that $D\neq a$. In this case there exists $\alpha^{(i)}\in {}_{0}\mathcal{B}_{\omega}-Z$ such that $\rho\|\alpha^{(i)}$ is not zero in $\HH^{2}(A)$.
Observe that $H:=\langle\rho\|\alpha^{(i)}\rangle$ is a representation of $\HH^{1}(A)$ contained in $V_{\rho}$, since given $\alpha\|c\in \HH^{1}(A)$, 
Eq. (\ref{eq:corchete}) implies that,

\[[\alpha\|c, \rho\|\alpha^{(i)}]=\delta_{\alpha, \overleftarrow{{}_{1}(\alpha^{(i)}) }}\rho\|\alpha^{(i)}-\delta_{\alpha,\overleftarrow{{}_{1}\rho}}\rho\|\alpha^{(i)},\]
but $a$ is positive and $V_{\rho}$ is irreducible so this leads us to a contradiction and we conclude that $D=a$.
\end{proof}

\subsection{$\HH^{i}(A),\ i\geq 3$ as representation of $\HH^{1}(A)$}

\label{Hncomorep}

Fix $i>2$. Given $\alpha\in Q_{1}$ and $b$ a non zero path in $A$ such that $0=s(\alpha)=s(b)$ and $t(\alpha)=t(b)$,
consider $\alpha\|b\in \HH^{1}(A)$. Also, 
given an $(i-1)$-ambiguity $u$ from $0$ to $\omega$ and a path $c$ in ${}_{0}\mathcal{B}_{\omega}$ consider 
that, as usual,  we will identify with its class $u\|c\in \HH^{i}(A)$. 
Next we will compute the Gerstenhaber bracket of these two elements. Given an $(i-1)$-ambiguity $w$,
\begin{align*}
[\alpha\|b, u\|c](w)&=\varphi_{i}^{*}[\eta_{1}^{*}(\alpha\|b), \eta_{i}^{*}(u\|c)](w)\\&=[\eta_{1}^{*}(\alpha\|b), \eta_{i}^{*}(u\|c)]\varphi_{i}(1\otimes w\otimes 1)\\&=\eta_{1}^{*}(\alpha\|b)\circ \eta_{i}^{*}(u\|c)(\varphi_{i}(1\otimes w\otimes 1))-\eta_{i}^{*}(u\|c)\circ \eta_{1}^{*}(\alpha\|b)(\varphi_{i}(1\otimes w\otimes 1)).
\end{align*}

The first term is equal to
\begin{align*}
\eta_{1}^{*}(\alpha\|b)(\eta_{i}^{*}(u\|c)(\varphi_{i}(1\otimes w\otimes 1)))&=\eta_{1}^{*}(\alpha\|b)((u\|c)\eta_{i}(\varphi_{i}(1\otimes w\otimes 1))\\&=\eta_{1}^{*}(\alpha\|b)(u\|c)(1\otimes w\otimes 1)\\&=\delta_{u,w}\eta_{1}^{*}(\alpha\|b)(c)\\&=\delta_{u,w}(\alpha\|b)\sum_{c}c^{(1)}\otimes \overleftarrow{c^{(2)}}\otimes c^{(3)}\\&=\left\{
                                                                                                                                     \begin{array}{ll}
                                                                                                                                       \delta_{w,u}c^{\alpha\|b}, & \hbox{if $\alpha=\overleftarrow{{}_{1}c}$;} \\
                                                                                                                                      0 , & \hbox{otherwise.}
                                                                                                                                     \end{array}
                                                                                                                                   \right.
\end{align*}

Observe that in the third equality we have used Proposition \ref{lema2}.
Let us compute the second term. By Proposition \ref{lema1}, we know that $\varphi_{i}(1\otimes w\otimes 1)$ has the following form,
\[\sum_{a^{(1)}\dots a^{(i+1)}=w}1\otimes \overline{a^{(1)}}\otimes\dots \otimes \overline{a^{(i)}}\otimes a^{(i+1)}.\]
Thus,\\
$\displaystyle \eta_{i}^{*}(u\|c)\circ \eta_{1}^{*}(\alpha\|b)(\varphi_{i}(1\otimes w\otimes 1))=\sum_{a^{(1)}\dots a^{(i+1)}=w}\eta_{i}^{*}(u\|c)\circ \eta_{1}^{*}(\alpha\|b)(1\otimes \overline{a^{(1)}}\otimes\dots \otimes \overline{a^{(i)}}\otimes a^{(i+1)})=\sum_{a^{(1)}\dots a^{(i+1)}=w}\eta_{i}^{*}(u\|c)\sum_{j=1}^{i}((\overline{a^{(1)}}\otimes\dots\otimes\eta_{1}^{*}(\alpha\|b)(\overline{a^{(j)}})\otimes\dots \otimes \overline{a^{(i)}}) a^{(i+1)}).$\\
Since $\alpha$ is the first arrow of some branch, we know that $\eta_{1}^{*}(\alpha\|b)(\overline{a^{(j)}})=0$ for $j\neq 1$, so the last expression equals 
\[\sum_{a^{(1)}\dots a^{(i+1)}=w}\eta_{i}^{*}(u\|c)((\eta_{1}^{*}(\alpha\|b)(\overline{a^{(1)}})\otimes\dots \otimes \overline{a^{(i)}}) a^{(i+1)})\] 
which, by definition of $\eta^{\ast}$ and using again that $\alpha$ is the first arrow of some branch, is
\[\delta_{\alpha, \overleftarrow{{}_{1}\omega}}\sum_{a^{(1)}\dots a^{(i+1)}=w}\eta_{i}^{*}(u\|c)(( \overline{a^{(1)}}\otimes\dots \otimes \overline{a^{(i)}}) a^{(i+1)}).\]
By Proposition \ref{prop4.6} we have that,
\begin{align*}
\delta_{\alpha, \overleftarrow{{}_{1}\omega}}\!\! \! \!\! \! \sum_{a^{(1)}\dots a^{(i+1)}=w} \! \! \! \! \!\! \! \!\! \! \eta_{i}^{*}(u\|c)(( \overline{a^{(1)}}\otimes\dots \otimes \overline{a^{(i)}}) a^{(i+1)})&=\delta_{\alpha, \overleftarrow{{}_{1}\omega}}(u\|c)\eta_{i}(\! \!\! \!\! \! \! \!\! \!  \sum_{a^{(1)}\dots a^{(i+1)}=w}\! \!\! \! \! \! \!\!\! \! (1\otimes\overline{ a^{(1)}}\otimes\dots \otimes \overline{a^{(i)}}\otimes a^{(i+1)}))\\&=\delta_{\alpha, \overleftarrow{{}_{1}\omega}}(u\|c)(1\otimes w\otimes 1)\\&=\delta_{u,w}\delta_{\alpha, \overleftarrow{{}_{1}\omega}}c.
\end{align*}

Summarising, given $\alpha\|b\in \HH^{1}(A)$ and $u\|c\in \HH^{i}(A)$:
\[[\alpha\|b, u\|c]=\delta_{\overleftarrow{{}_{1}c},\alpha}u\|c^{\alpha\|b}-\delta_{\overleftarrow{{}_{1}u},\alpha}u\|c.\]

\noindent Theorems \ref{descH2sla} and \ref{descH2} can be adapted to the general case, that is, $i\geq 2$. 
We state them now and we omit the proofs since they are analogous to the case $i=2$.

\begin{theorem}\label{mainHn}
Consider a toupie algebra $A$ with $a> 0$. The decomposition of $\HH^{i}(A)$ as a representation of $sl_{a}(\mathbb{C})$ is:
\[\HH^{i}(A)\simeq\bigoplus_{i=1}^{D-a} V_{0}~~\oplus \bigoplus_{u:s(u)=0, t(u)=\omega} V,\]
where $V_{0}\simeq \mathbb{C}$ is the trivial representation and $V$ is the standard representation of $sl_{a}(\mathbb{C})$.
\end{theorem}
\begin{remark}
Given an $(i-1)$-ambiguity $u$ from $0$ to $\omega$, the $\mathds{k}$-vector subspace of 
$\HH^{i}(A)$ $V_{u}=\langle u\|\alpha^{(k)} ~~|~~\alpha^{(k)}\in {}_{0}\mathcal{B}_{\omega} \rangle$ is a Lie subrepresentation of $\HH^{1}(A)$.
\end{remark}
\begin{theorem}
Let $A$ be a toupie algebra with $a>0$. The decomposition of $\HH^{i}(A)$ as a Lie representation of  $\HH^{1}(A)$ is the following:
\[\HH^{i}(A)=\bigoplus_{u:s(u)=0, t(u)=\omega}V_{u}\]
where $V_{u}=\langle u\|\alpha^{(k)} ~~|~~\alpha^{(k)}\in{}_{0}\mathcal{B}_{\omega}\rangle $. Besides, note that for every $(i-1)$-ambiguity $u$ from $0$ to $\omega$, 
the module $V_{u}$ is indecomposable.
The representation $V_{u}$ is irreducible if and only if $D=a$.
\end{theorem}

We end this article with an example, for which we compute the whole structure.

\newpage

\begin{example}
Consider the toupie algebra $A=\mathds{k}Q/I$ where $Q$ is the quiver bellow with $\#Q_{0}=13$ and $\#Q_{1}=15$:
\begin{center}
\begin{figure}[hhh]
\includegraphics[scale=0.62]{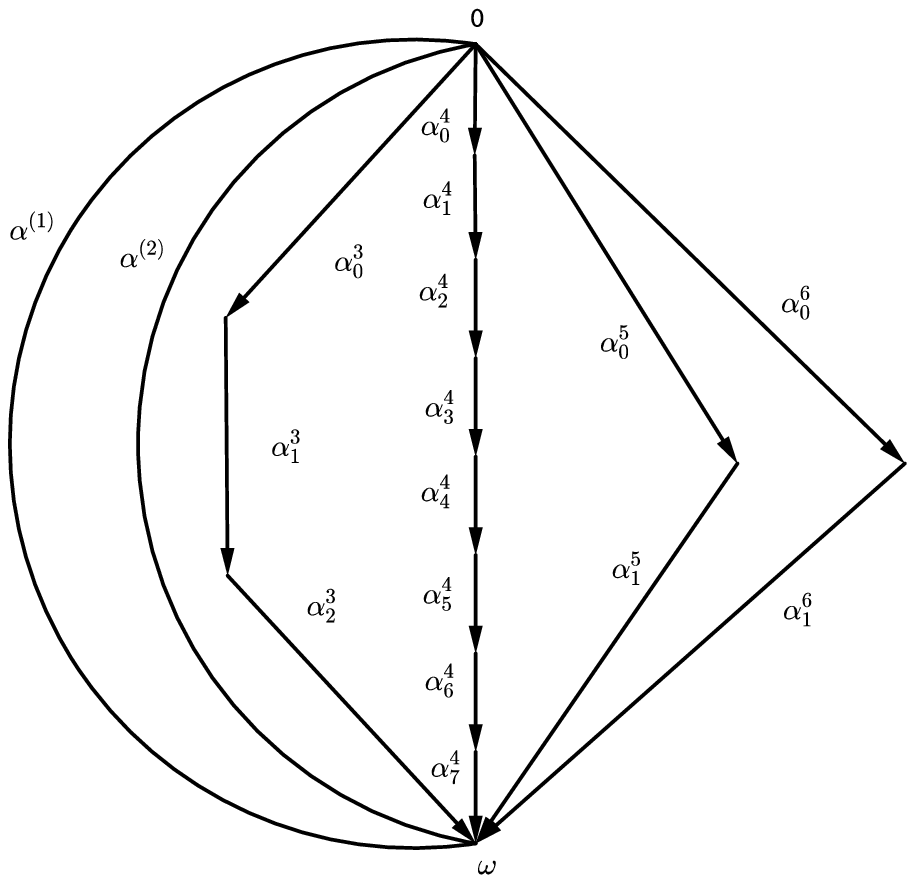}
\end{figure}
\end{center}
\noindent Let $\rho'_{1}=\alpha_{0}^{5}\alpha_{1}^{5}-\alpha_{0}^{6}\alpha_{1}^{6}$ and $\sigma_{i}=\alpha_{i}^{4}\alpha_{i+1}^{4}\alpha_{i+2}^{4}\alpha_{i+3}^{4}$ with $i=0, \dots, 4$
and $I=\langle\rho'_{1}, \sigma_{i}: i=0, \dots ,4\rangle$.
Note that there are two arrows from $0$ to $\omega$, there is one branch  not involved in any relation, one branch  with monomial relations and two branches involved in a non 
monomial relation. This means that $D=4$, $a=2$, $l=1$, $m=1$ and $n=2$.

There are four $2$-ambiguities, $\alpha_{0}^{4}\alpha_{1}^{4}\alpha_{2}^{4}\alpha_{3}^{4}\alpha_{4}^{4}$, $\alpha_{1}^{4}\alpha_{2}^{4}\alpha_{3}^{4}\alpha_{4}^{4}\alpha_{5}^{4}$, 
$\alpha_{2}^{4}\alpha_{3}^{4}\alpha_{4}^{4}\alpha_{5}^{4}\alpha_{6}^{4}$  and $\alpha_{3}^{4}\alpha_{4}^{4}\alpha_{5}^{4}\alpha_{6}^{4}\alpha_{7}^{4}$ and one $3$-ambiguity $\alpha^{(4)}$, for $j\geq 4$ the set of $j$-ambiguities is empty.

The Hochschild cohomology is the following: 
\begin{itemize}
\item $\displaystyle HH^{0}(A)=\langle\sum_{i\in Q_{0}}e_{i}\| e_{i}\rangle$

\item 
$HH^{1}(A)$ we obtain that
$HH^{1}(A)=\langle y_{4}, w_{12}, w_{21}, x_{2}, z_{13}, z_{23}, z_{16}, z_{26}, t_{1}, t_{2} \rangle$, where we have used the previously defined notation.
\item $HH^{2}(A)=\langle\overline{\rho'_{1}\|\alpha^{(1)}}, \overline{\rho'_{1}\|\alpha^{(2)}}, \overline{\rho'_{1}\|\alpha^{(3)}}\rangle$ 
\item $HH^{3}(A)=0$ 
\item $HH^{4}(A)=\langle \overline{\alpha^{(4)}\|\alpha^{(1)}}, \overline{\alpha^{(4)}\|\alpha^{(2)}}, \overline{\alpha^{(4)}\|\alpha^{(3)}}, \overline{\alpha^{(4)}\|\alpha^{(6)}} \rangle$
\item $HH^{i}(A)=0$ for $i>4$.
\end{itemize}
The decomposition of the Lie algebra $HH^{1}(A)$ is the following:
$$HH^{1}(A)\cong \langle y_{4}\rangle\oplus sl_{2}(\mathbb{C})\ltimes (\langle t_{1}, t_{2}\rangle\ltimes \langle z_{13}, z_{23}, z_{16}, z_{26}\rangle ) $$

The non-null Gerstenhaber brackets of $HH^{1}(A)$ with $HH^{2}(A)$ are the following:
 \begin{multicols}{3}
 \item $[w_{12},\overline{\rho'_{1}\|\alpha^{(1)}} ]=\overline{\rho'_{1}\|\alpha^{(2)}}$
\item $[x_{2},\overline{\rho'_{1}\|\alpha^{(1)}}]=-\overline{\rho'_{1}\|\alpha^{(1)}}$
\item $[z_{13}, \overline{\rho'_{1}\|\alpha^{(1)}}]=\overline{\rho'_{1}\|\alpha^{(3)}}$
\item $[t_{2}, \overline{\rho'_{1}\|\alpha^{(1)}}]=-\overline{\rho'_{1}\|\alpha^{(1)}}$
\item $[w_{21},\overline{\rho'_{1}\|\alpha^{(2)}}]=\overline{\rho'_{1}\|\alpha^{(1)}}$
\item $[x_{2}, \overline{\rho'_{1}\|\alpha^{(2)}}]=\overline{\rho'_{1}\|\alpha^{(2)}}$
\item $[z_{23}, \overline{\rho'_{1}\|\alpha^{(2)}}]=\overline{\rho'_{1}\|\alpha^{(3)}}$
\item $[t_{2}, \overline{\rho'_{1}\|\alpha^{(2)}}]=-\overline{\rho'_{1}\|\alpha^{(2)}}$
\item $[t_{1}, \overline{\rho'_{1}\|\alpha^{(3)}}]=\overline{\rho'_{1}\|\alpha^{(3)}}$
\item $[t_{2}, \overline{\rho'_{1}\|\alpha^{(3)}}]=-\overline{\rho'_{1}\|\alpha^{(3)}}$
\end{multicols}
\newpage
The non-null Gerstenhaber brackets of $HH^{1}(A)$ with $HH^{4}(A)$ are the following:
\begin{multicols}{3}
\item $[y_{4}, \overline{\alpha^{(4)}\|\alpha^{(1)}}]=-\overline{\alpha^{(4)}\|\alpha^{(1)}}$
\item $[w_{12}, \overline{\alpha^{(4)}\|\alpha^{(1)}}]=\overline{\alpha^{(4)}\|\alpha^{(2)}}$
\item $[x_{2}, \overline{\alpha^{(4)}\|\alpha^{(1)}}]=-\overline{\alpha^{(4)}\|\alpha^{(1)}}$
\item $[z_{13}, \overline{\alpha^{(4)}\|\alpha^{(1)}}]=\overline{\alpha^{(4)}\|\alpha^{(3)}}$
\item $[z_{16}, \overline{\alpha^{(4)}\|\alpha^{(1)}}]=\overline{\alpha^{(4)}\|\alpha^{(6)}}$
\item $[y_{4}, \overline{\alpha^{(4)}\|\alpha^{(2)}}]=-\overline{\alpha^{(4)}\|\alpha^{(2)}}$
\item $[w_{21}, \overline{\alpha^{(4)}\|\alpha^{(2)}}]=\overline{\alpha^{(4)}\|\alpha^{(1)}}$
\item $[x_{2}, \overline{\alpha^{(4)}\|\alpha^{(2)}}]=\overline{\alpha^{(4)}\|\alpha^{(2)}}$
\item $[z_{23}, \overline{\alpha^{(4)}\|\alpha^{(2)}}]=\overline{\alpha^{(4)}\|\alpha^{(3)}}$
\item $[z_{26}, \overline{\alpha^{(4)}\|\alpha^{(2)}}]=\overline{\alpha^{(4)}\|\alpha^{(6)}}$
\item $[y_{4},\overline{\alpha^{(4)}\|\alpha^{(3)}}]=-\overline{\alpha^{(4)}\|\alpha^{(3)}}$
\item $[t_{1}, \overline{\alpha^{(4)}\|\alpha^{(3)}}]=\overline{\alpha^{(4)}\|\alpha^{(3)}}$
\item $[y_{4},\overline{\alpha^{(4)}\|\alpha^{(6)}}]=-\overline{\alpha^{(4)}\|\alpha^{(6)}}$
 \item $[t_{2}, \overline{\alpha^{(4)}\|\alpha^{(6)}}]=\overline{\alpha^{(4)}\|\alpha^{(6)}}$
\end{multicols}

\end{example}

\end{document}